\documentclass[11pt]{article}
\usepackage{latexsym}
\usepackage{amsmath,lipsum}
\usepackage{amssymb,amsfonts,amscd}
\usepackage{graphicx}
\usepackage{multirow}
\usepackage{mathabx} % for widecheck
\usepackage{hyperref}
\usepackage{color}
\usepackage{subcaption}
\usepackage{cite}
\usepackage{algorithm}
\usepackage{algorithmic}
\usepackage[dvipsnames]{xcolor}
\usepackage{changes}
\usepackage{hhline} 
\usepackage{tabularx}
\usepackage{array}
\newcolumntype{P}[1]{>{\centering\arraybackslash}p{#1}}
\usepackage{siunitx}

\newtheorem{theorem}{Theorem}[section]

\newtheorem{rem}[theorem]{Remark}

\newtheorem{example}[theorem]{Example}

\newcommand\qed{{\hspace*{\fill}$\Box$\vskip12pt plus 1pt}}

\newenvironment{proof}{{\noindent\bf Proof.\ }}{\qed}
\newenvironment{remark}{\begin{rem}\em}{\end{rem}}

\newcommand\bN{{\mathbb N}}

\newcommand\bR{{\mathbb R}}

\begin{document}
\title{A numerical method for solving elliptic equations \\
on real closed algebraic curves and surfaces}

\author{
Wenrui Hao\thanks{Department of Mathematics, Penn State University, University Park, PA 16802 
(wxh64@psu.edu). } \and
Jonathan D. Hauenstein\thanks{Department of Applied and Computational Mathematics and Statistics,
University of Notre Dame, Notre Dame, IN 46556 (hauenstein@nd.edu, \url{www.nd.edu/\~jhauenst}).
This author was supported in part by 
the National Science Foundation CCF-1812746.
}
\and
Margaret H. Regan\thanks{Department of Mathematics and Computer Science,
College of the Holy Cross, Worcester, MA 01610 (mregan@holycross.edu,
\url{www.margaretregan.com}).}
\and
Tingting Tang\thanks{Department of Mathematics and Statistics,
San Diego State University, San Diego, CA 92182 (ttang2@sdsu.edu, \url{sites.google.com/sdsu.edu/mathtingting-tang/home}).}
}
%\author{Wenrui Hao         \and
%        Jonathan D. Hauenstein \and %etc.
%        Margaret H. Regan \and
%        Tingting Tang
%}

%\institute{Wenrui Hao \at
%              Department of Mathematics, Penn State University, University Park, PA, USA 16802 \email{wxh64@psu.edu}. 
%           Jonathan D. Hauenstein \at
%              Department of Applied and Computational Mathematics and Statistics,
%University of Notre Dame, Notre Dame, IN, USA 46556 \email{hauenstein@nd.edu}.%, \url{www.nd.edu/\~jhauenst}). This author was supported in part by the National Science Foundation CCF-1812746.
%    \and
 %   Margaret H. Regan \at Department of Mathematics and Computer Science, College of the Holy Cross, Worcester, MA, USA 01610 \email{mregan@holycross.edu}. %, \url{www.margaretregan.com}), Corresponding author.
  %  \and
   % Tingting Tang \at Department of Mathematics and Statistics,
% San Diego State University, San Diego, CA, USA 92182 \email{ttang2@sdsu.edu}. %\url{sites.google.com/sdsu.edu/mathtingting-tang/home}).
%}


\maketitle

\begin{abstract}
\noindent 
There are many numerical methods for solving partial different equations (PDEs) on manifolds such as classical implicit, finite difference, finite element, and isogeometric analysis methods which aim at improving the interoperability between finite element method and computer aided design (CAD) software. However, these approaches have difficulty 
when the domain has singularities
since the solution
at the singularity may be multivalued.
This paper develops a novel numerical approach to solve elliptic PDEs on real, closed, connected, orientable, and almost smooth algebraic curves and surfaces.
Our method integrates numerical algebraic geometry, differential geometry, and a finite difference scheme which is demonstrated on
several examples.


\end{abstract}

% 65N06 PDE BVP finite difference methods
% 65H10 Numerical algebraic geometry
% 68W30 Symbolic computation and algebraic computation

\section{Introduction}\label{Sec:Intro}
Advances in fluid dynamics, biology, material science, and other disciplines have promoted the study of partial differential equations (PDEs) 
defined on various manifolds. Numerous numerical methods have been developed to solve these PDEs, such as classical implicit~\cite{BCOS,BSCO,OS}, finite \mbox{difference \cite{MDS,T1,X1}}, 
finite element~\cite{DE,LC,RWP}, and parameterization methods~\cite{SK,WLGHCT}. 
In this paper, we specifically consider linear elliptic PDEs
defined on closed algebraic curves and surfaces, which are described
implicitly as the solution to a system of polynomial equations.
We consider the well-posedness of the problem when the domain has singularities corresponding
to problems in which variational methods can not be applied. 
In particular, when the domain is a real closed algebraic curve, we 
can always reduce the problem to solving an ordinary differential equation~(ODE) described in terms of the arc length.  Numerically,
we can construct a meshing of the curve which is uniform in arc length
via numerical algebraic geometry \cite{BHSW:Bertini,BertiniReal}. 
Such an approach is not limited to smooth curves nor when an {\em a priori}
global parameterization of the curve is known.
From the meshing, we introduce a local tangential parameterization 
and embed it in a finite difference scheme to numerically solve the problem. A similar approach is extended to real closed algebraic surfaces
which are almost smooth, i.e., have at most finitely many singularities. 
%\deleted{Recent advances in devices, imaging, and other biomedical technologies in have resulted in the exponential growth of biomedical data, and have started the era of ``Big Data" in health and biomedicine \cite{YTHWH,KCGKLG}. In particular, advanced data acquisition technology, such as CT, MRI, and 3D camera imaging,provides three-dimensional or even higher dimensional geometric data. It then becomes natural to consider traditional mathematical models, especially systems of partial differential equations (PDEs), on a general manifold generated by the geometric data. There are many current numerical methods in this direction such as classical implicit \cite{BCOS,BSCO,OS}, finite difference \cite{MDS,T1,X1}, finite element \cite{DE,LC,RWP}, and parameterization methods \cite{SK,WLGHCT}. \textcolor{cyan}{This paragraph does not quite fit as motivation of the type of problems we are working on. Our focus is what happens when the domain has positive or high co-dimension or has singularities. In particular, when combined with numerical algebraic geometry, we can sample the domain much more efficiently. Using differential geometry, we then reduce the dimension of the problem to the dimension of the domain, and finally numerically solve the problem. This may be a good place to motivate the use of numerical algebraic geometry for sampling.} Our proposed method uses numerical algebraic geometry and local parameterization methods to solve elliptic PDEs on real, closed algebraic curves and surfaces.}

The linear elliptic PDEs under consideration have
the form
\begin{equation}\label{Eq: ellip_pde}
	-\Delta u + c\cdot u = f \qquad \hbox{~on~} \Omega 
\end{equation} 
where $\Omega$ is a closed, connected, and orientable $d$-dimensional 
algebraic set in $\mathbb{R}^n$ where $0<d<n$. 
Thus, $\Omega$ is described by the solution set of a system
of polynomial equations $F = 0$ on $\bR^n$.
Curves have $d = 1$ while surfaces have $d = 2$.
For example, the unit circle in $\bR^2$ 
as shown in Fig.~\ref{Fig:Lemniscate}(a)
is a curve defined by the 
solution set of the polynomial equation $x^2 + y^2 - 1 = 0$
while the unit sphere in $\bR^3$ is a surface defined by
the solution set of the polynomial equation $x^2 + y^2 + z^2 -1 = 0$.
The operator $\Delta$ is the Laplace-Beltrami operator on $\Omega$
while $c$ and $f$ are functions independent of $u$.
With this setup, the dimension of the tangent space at each point in $\Omega$
is at least $d$.  The smooth points of $\Omega$
are the points where the dimension of the tangent space is equal to $d$
while the singular points are those where the dimension of the tangent space is larger than $d$.  For curves ($d = 1$),
the number of singular points is always finite, e.g., the lemniscate of Gerono
showed in Fig.~\ref{Fig:Lemniscate}(b)
has one singular point. 
We only consider surfaces ($d = 2$)
where the number of singular points is finite, called {\em almost smooth} surfaces. 
The horn torus shown in Fig.~\ref{Fig:Lemniscate}(c)
is an almost smooth surface with one
singular point
while the Whitney umbrella
shown in Fig.~\ref{Fig:Lemniscate}(d)
is not an almost smooth surface
since it has a line of singularities.

\begin{figure}[!b]
    \centering
    $\begin{array}{ccccccc}
    \includegraphics[scale=0.04]{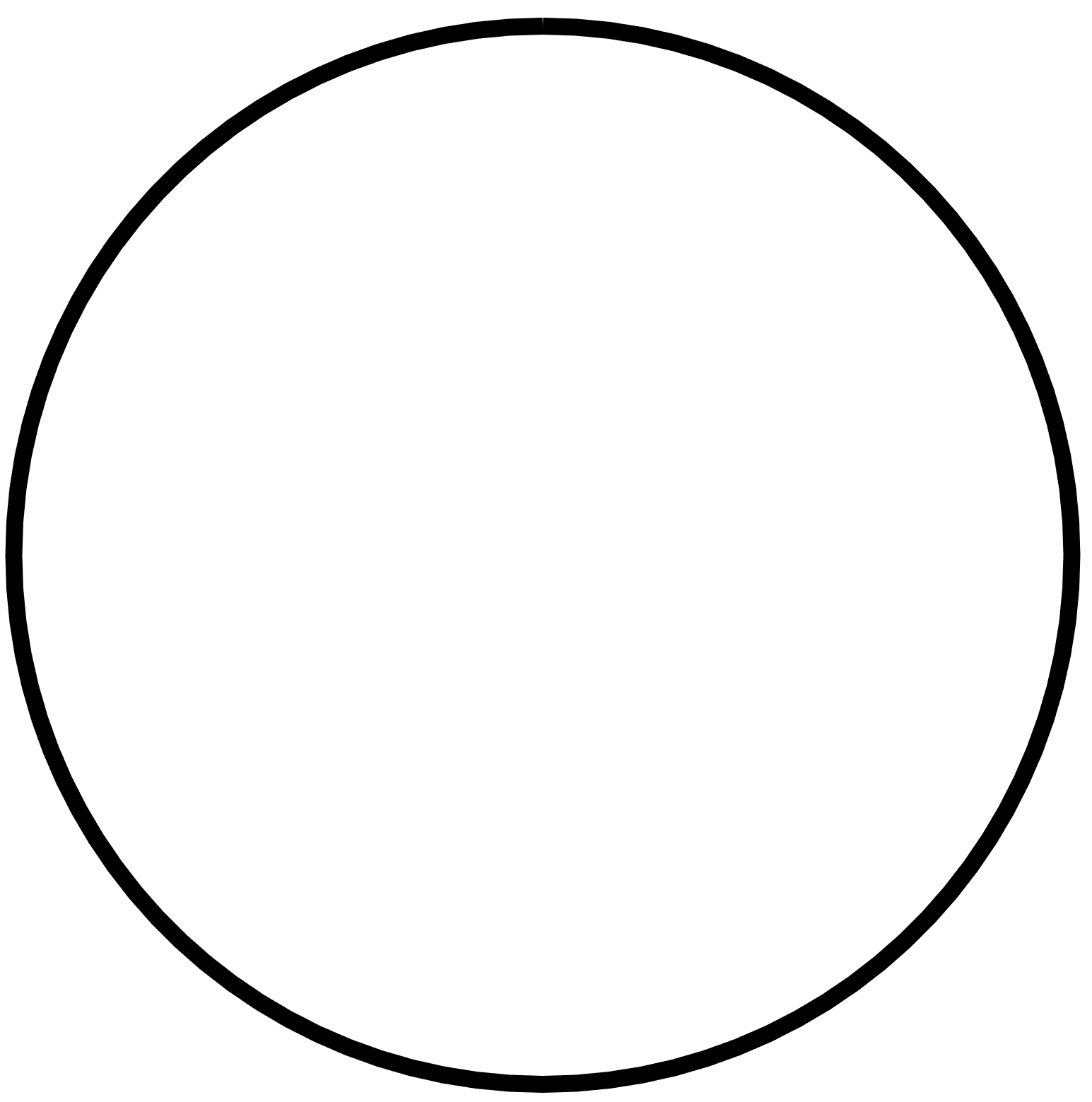} & &
    \includegraphics[scale=0.05]{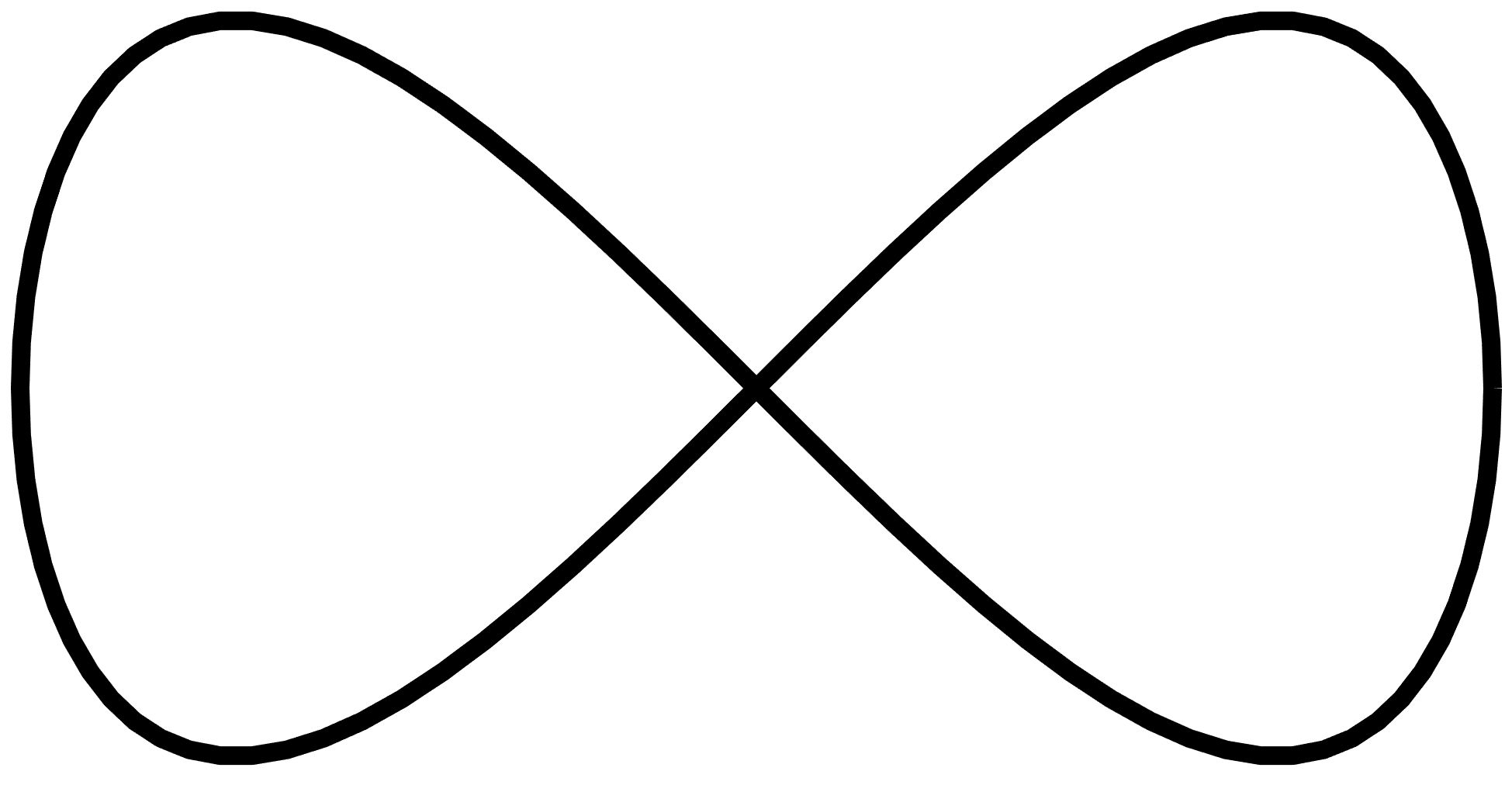} & &
    \includegraphics[scale=0.05]{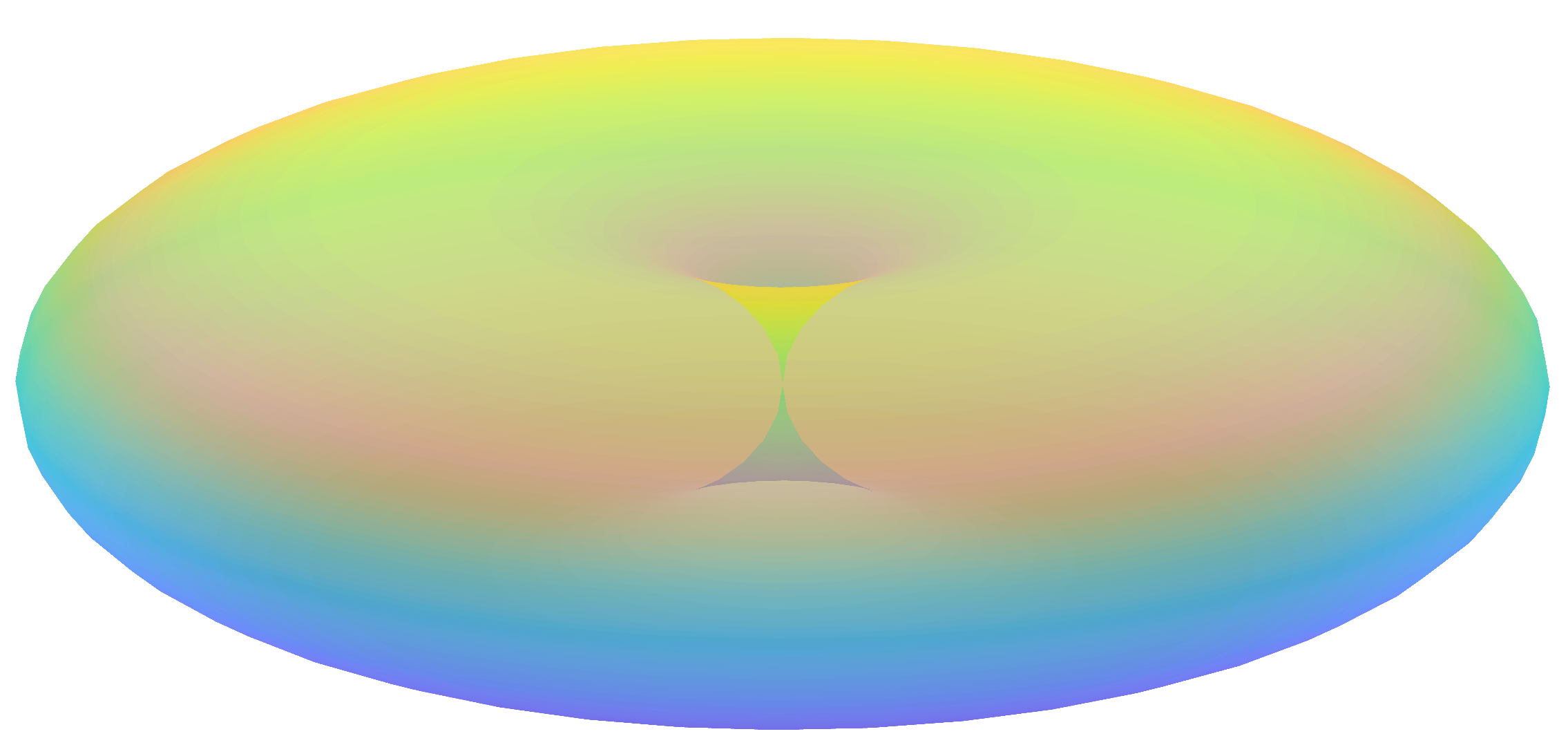} & &
    \includegraphics[scale=0.05]{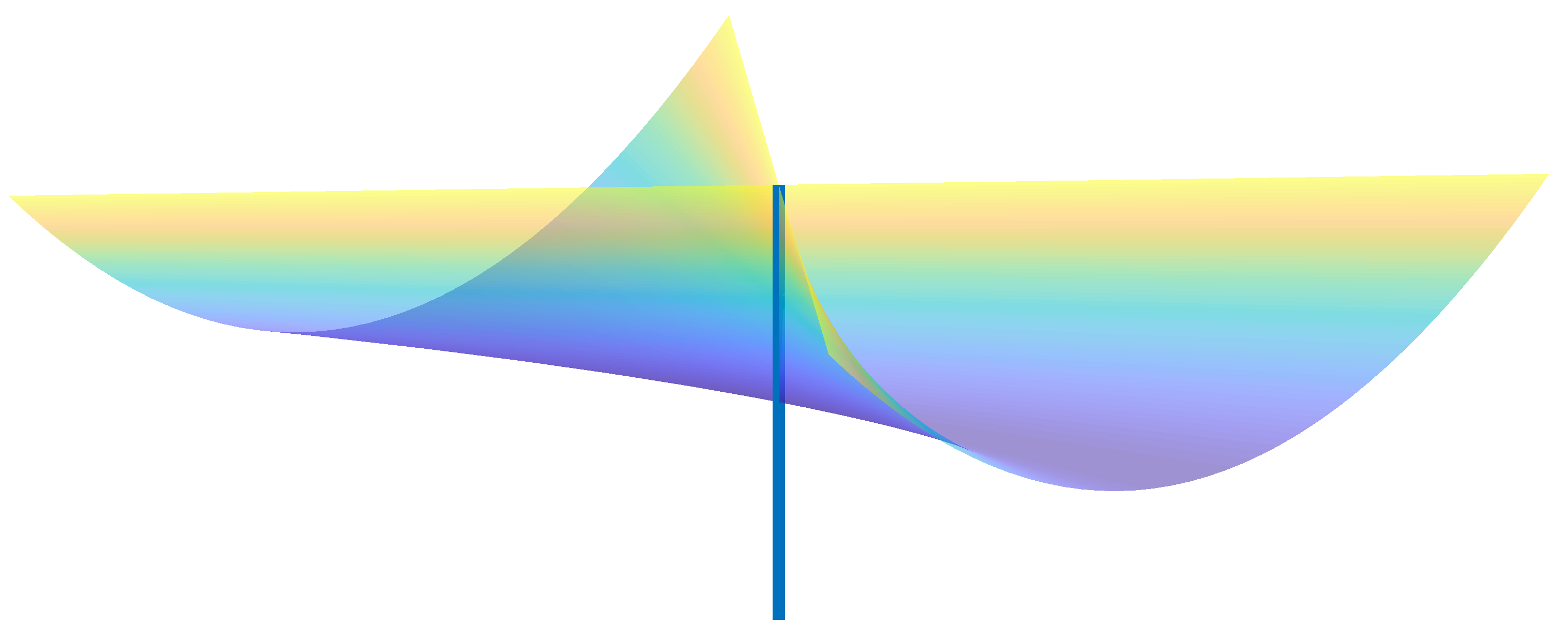} \\
    \hbox{(a)} & & \hbox{(b)} & & \hbox{(c)} & & \hbox{(d)} \\
    \end{array}$
    \caption{(a) circle, (b) lemniscate of Gerono, (c) horn torus, and (d) Whitney umbrella
    }
    \label{Fig:Lemniscate}
\end{figure}

For any $d$, if there are no singular points, then $\Omega$
is said to be smooth, i.e., a manifold, and there are many existing numerical methods, 
e.g., \cite{BHLM,BCOS,BSCO,LAGR,OS,MDS,T1,X1,DE,LC,RWP,SK,WLGHCT},
for solving \eqref{Eq: ellip_pde}.
For example, \cite{DE} considered finite element methods for solving
on triangulated surfaces and implicit surface methods using a level set description of the surface. 
Variational techniques for solving on smooth surfaces 
based on splines and non-uniform B-splines (NURBS) are reviewed in \cite{LAGR}. Recently,  \cite{BHLM} established the theoretical framework to analyze cut finite element methods for the Laplace-Beltrami operator defined on a manifold.
These methods focus on smooth surfaces which either can be parameterized or implicitly represented by level sets. In the case of the implicit surface methods, a discretization of the space where the manifold is embedded in is required, which can be inefficient when the codimension, i.e., $n - d$, is high.  %Section~\ref{Sec:weak_solution} outlines the existence of a unique weak solution to Eq.~\ref{Eq: ellip_pde} given $\Omega$ is a almost-Riemann structure. In addition, we present a finite difference scheme coupled with a local tangential parameterization method in Sections~\ref{sec:local_param} and~\ref{Sec:2d} for smooth curves and $C^1$ hypersurfaces, respectively.

To the best of our knowledge,
little to no studies have been done to investigate the existence of a theoretical or numerical solution on curves with singularities. 
One possible reason for this is that the solution $u$ to \eqref{Eq: ellip_pde} need not take a
single value at a singularity
of $\Omega$ due to the presence
of multiple local irreducible
components at a singularity,
e.g., the lemniscate of Gerono
shown in Fig.~\ref{Fig:Lemniscate}(b)
has two local irreducible components
at the singular point.
As an illustration, 
Figure~\ref{fig:IntroPlot}
shows the solutions to 
the following two problems
\begin{equation}\label{eq:IntroEx}
\begin{array}{cc}
(a)
-\Delta u + \left(
\pi - \frac{4x_1^2 + 4x_2^2 - 3}{8x_1^2x_2^2 + 16x_2^4 - 3x_1^2 - 17x_2^2 + 4}\right)
\cdot u = \pi\cdot x_1
\hbox{~on~} \Omega
 & 
 (b)
 -\Delta u + u = x_1^2 + x_1x_2 - 1
\hbox{~on~} \Omega
 \end{array}
\end{equation}
where the domain 
is the lemniscate of Gerono 
shown in Fig.~\ref{Fig:Lemniscate}(b) and defined by
$$\Omega = \{(x_1,x_2)\in\bR^2~|~
x_1^4 - x_1^2 + x_2^2 = 0\}.
$$
The solution of the former
is $u = x_1$ 
which is univalued at the singularity $(0,0)$
while the solution of
the latter takes two
different values at $(0,0)$,
one along each of the two
local irreducible
components at $(0,0)$.
These problems will be further
considered in
Exs.~\ref{ex:LemniscateSolve}~and~\ref{ex:LemniscateSolveNum}, respectively.
Numerical algebraic geometry
will also be used to compute
the local irreducible components
\cite{LocalNID}
to ensure the proper structure
of the solution $u$ at the singularities.

\begin{figure}
\begin{center}
\begin{tabular}{ccc}
\includegraphics[scale = 0.06]{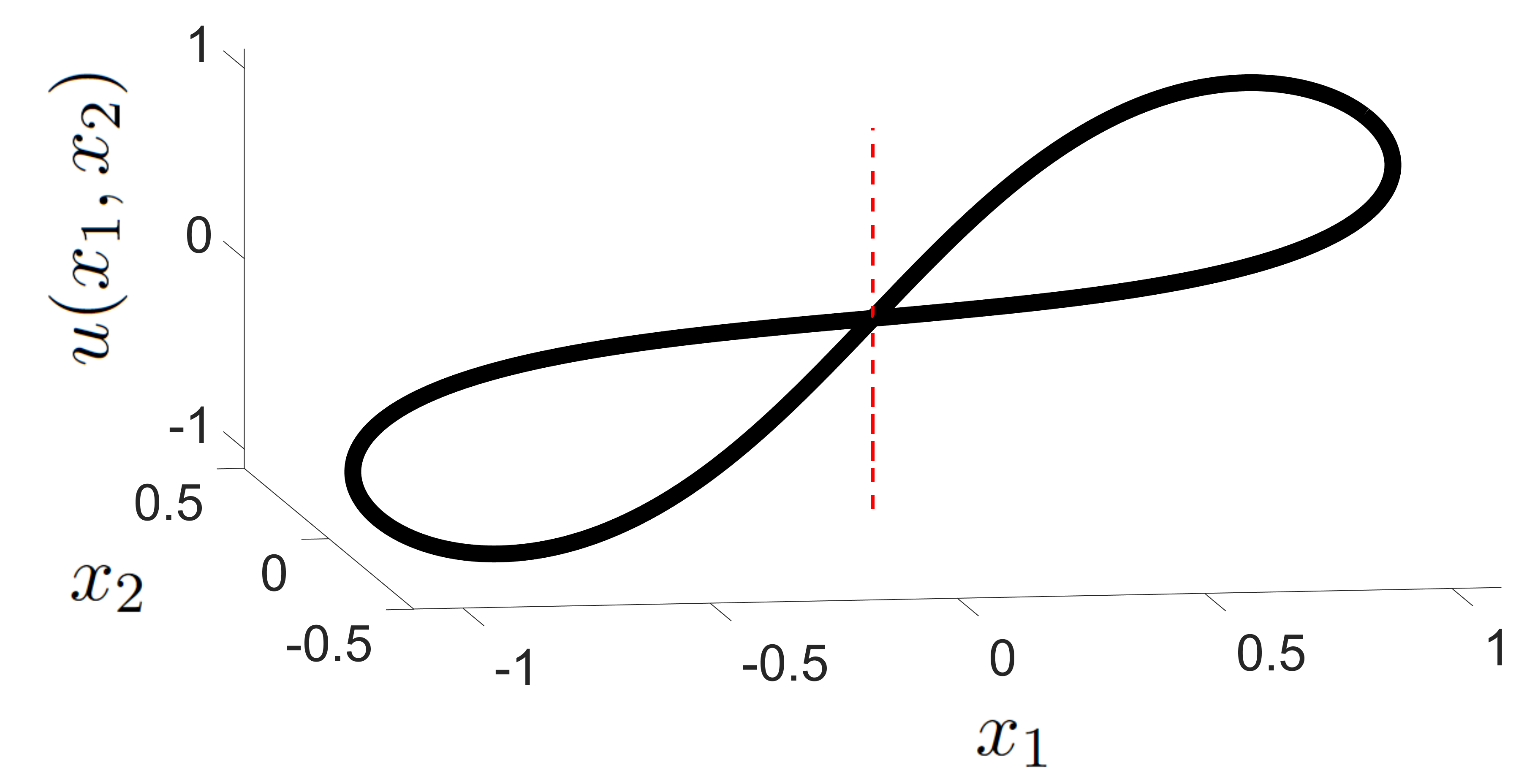}
&~~~~& \includegraphics[scale = 0.06]{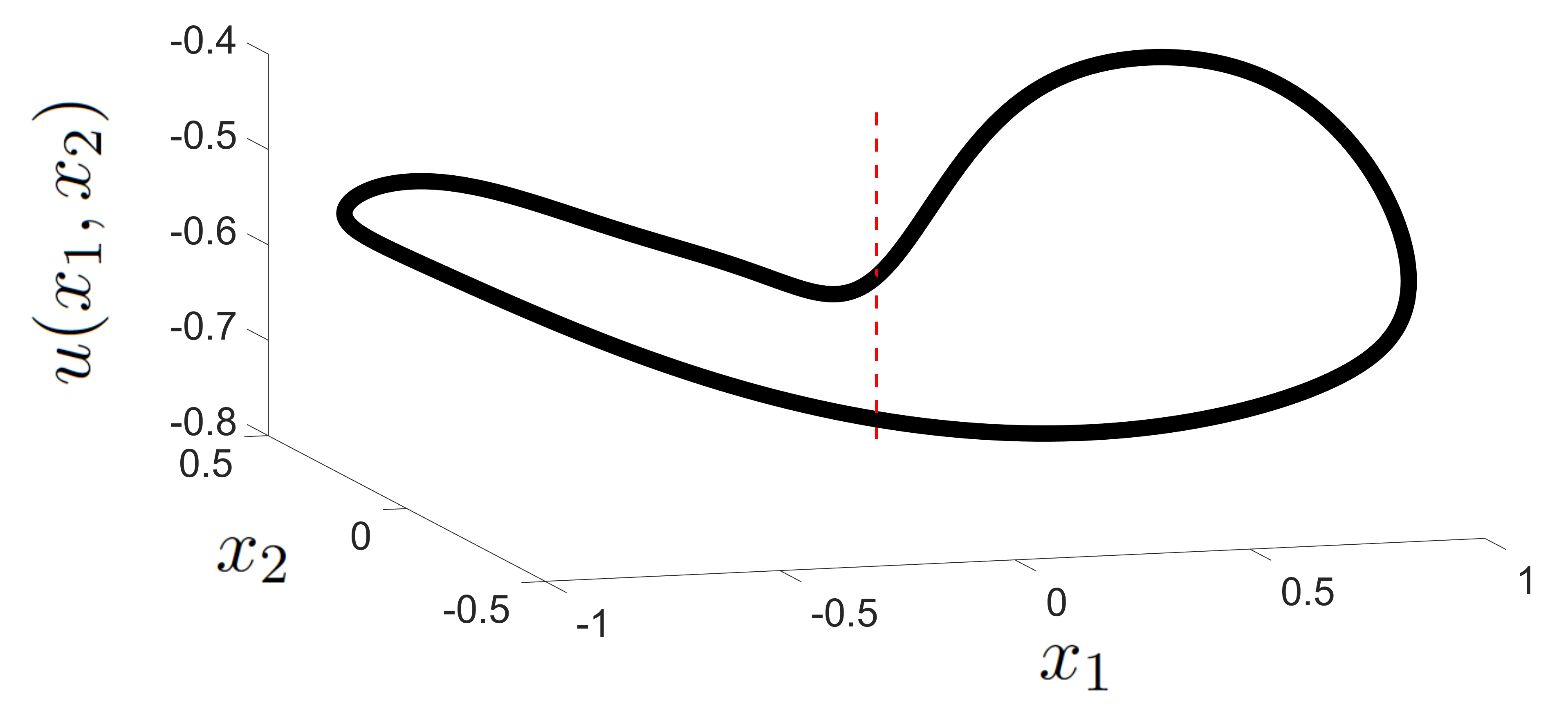}\\
(a) & & (b)
\end{tabular}
\end{center}
\caption{Solutions corresponding
to \eqref{eq:IntroEx}
on the lemniscate of Gerono
where the dashed line 
corresponds with $(x_1,x_2)=(0,0)$
showing the first is univalued
while the second is multivalued.
}
\label{fig:IntroPlot}
\end{figure}

The structure of the rest of the 
paper is as follows. Section~\ref{Sec:weak_solution} shows the existence and uniqueness of the solution to the elliptic problem \eqref{Eq: ellip_pde} under certain conditions
along with analysis when a global parameterization
is known.  Sections~\ref{Sec:1d} 
and~\ref{Sec:2d} describe a local tangential parameterization
at smooth points along with considering local irreducible components
at singularities.

\section{Global parameterization}\label{Sec:weak_solution}

%Formulation, well-posedness, and solving with a global parameterization}\label{Sec:weak_solution}

\subsection{Formulation}\label{sec:formulation}

For $k\in\bN\cup\{\infty\}$ and a connected set $D\subset\bR$,
let $C^k(D,\bR^n)$ consist of the functions \mbox{$\alpha:D\rightarrow\bR^n$} which are $k$-times continuously differentiable on $D$.
For $0\leq r\leq k$, let $\alpha^{(r)}(t)$ denote the $r^{\rm th}$
derivative of $\alpha$ at $t$.  A real algebraic curve $\Omega\subset\bR^n$ is called a {\em closed parametric $C^k$ curve} if 
there exists a closed interval $[a,b]\in\bR$ and a surjective 
map $X:[a,b]\rightarrow\Omega$
such that \hbox{$X\in C^k([a,b],\bR^n)$} with $X^{(r)}(a) = X^{(r)}(b)$
for all $0 \leq r \leq k$.  If $X$ is also a bijection between $[a,b)$
and~$\Omega$, then~$\Omega$ is {\em simple}.  
A function $h:\Omega\rightarrow\bR$ is {\em $k$-times continuously differentiable} 
on $\Omega$ if~\mbox{$h\circ X\in C^k([a,b],\bR)$}.

\begin{example}\label{ex:IllustrativeSimpleCurves}
The unit circle $\Omega=\{x_1^2+x_2^2=1\}\subset\bR^2$
shown in Fig.~\ref{Fig:Lemniscate}(a)
is a 
simple closed parametric $C^{\infty}$ curve.  
The surjective function $X:[0,2\pi]\mapsto \Omega$
defined by $X(\theta) = (\cos(\theta),\sin(\theta))$ 
is infinitely differentiable and bijects $[0,2\pi)$ onto $\Omega$.

The lemniscate of Gerono
$\Lambda = \{x_1^4-x_1^2+x_2^2=0\}\subset\bR^2$
shown in Fig.~\ref{Fig:Lemniscate}(b) is a closed parametric~$C^{\infty}$ curve since the surjection $Y:[0,2\pi]\mapsto\Lambda$ defined
by
$Y(\theta) = (\cos(\theta), \sin(2\theta)/2)$
is infinitely differentiable.  
The map $Y$ is not a bijection 
since $Y(\pi/2) = Y(3\pi/2) = (0,0)$
which is the self-intersection point.
Hence, $\Lambda$ is not a simple curve.
\end{example}

A real algebraic surface $\Omega\subset\bR^n$
is called a {\em closed parametric $C^k$ surface}
if, for every $x^*\in\Omega$, there exists a nonempty open connected set
$V\subset\bR^2$, an open set $U\subset\bR^n$ containing~$x^*$,
and a bijective map $X:V\rightarrow U\cap\Omega$ such that $X\in C^k(V,\bR^n)$
and the rank of the Jacobian matrix of $X$, denoted $JX$, at every point in $V$ is $2$.  
A function $h:\Omega\rightarrow\mathbb{R}$ is {\em $k$-times 
continuously differentiable} on $\Omega$ if $h\circ X\in C^k(V,\bR)$.

\begin{example}\label{ex:IllustrativeSimpleSurfaces}
The unit sphere $\Omega = \{x_1^2+x_2^2+x_3^2=1\}\subset\bR^3$ is a closed parameteric
$C^\infty$ surface.  
Due to rotational symmetry of the sphere,
we only need to consider one point,
say $x^*=(0,0,-1)$. 
As shown in Fig.~\ref{Fig:SurfacePlots},
one can take
$V = \{a_1^2+a_2^2<1/4\}\subset\bR^2$,
$U = \{x_1^2+x_2^2<1/4\}\subset\bR^3$
which clearly contains~$x^*$,
and bijective map $X:V\rightarrow U\cap \Omega$
defined by
$$X(a_1,a_2) = \left(a_1,a_2,-\sqrt{1-a_1^2-a_2^2}\right)$$
which is infinitely differentiable 
with full rank Jacobian matrix on $V$.

The Whitney umbrella $\Lambda = \{x_1^2 = x_2^2 x_3\}\subset\bR^3$ 
shown in Fig.~\ref{Fig:Lemniscate}(c)
is not 
a closed parametric~$C^k$ surface for any $k\in\bN\cup\{\infty\}$
since, for example, the surface $\Lambda$ near the point $(0,0,-1)$ is one-dimensional (called the ``handle'' of the Whitney umbrella).

\begin{figure}[!t]
    \centering
    $\begin{array}[t]{ccc}
    & \multirow{2}{*}{} & \\
    \includegraphics[scale=0.02]{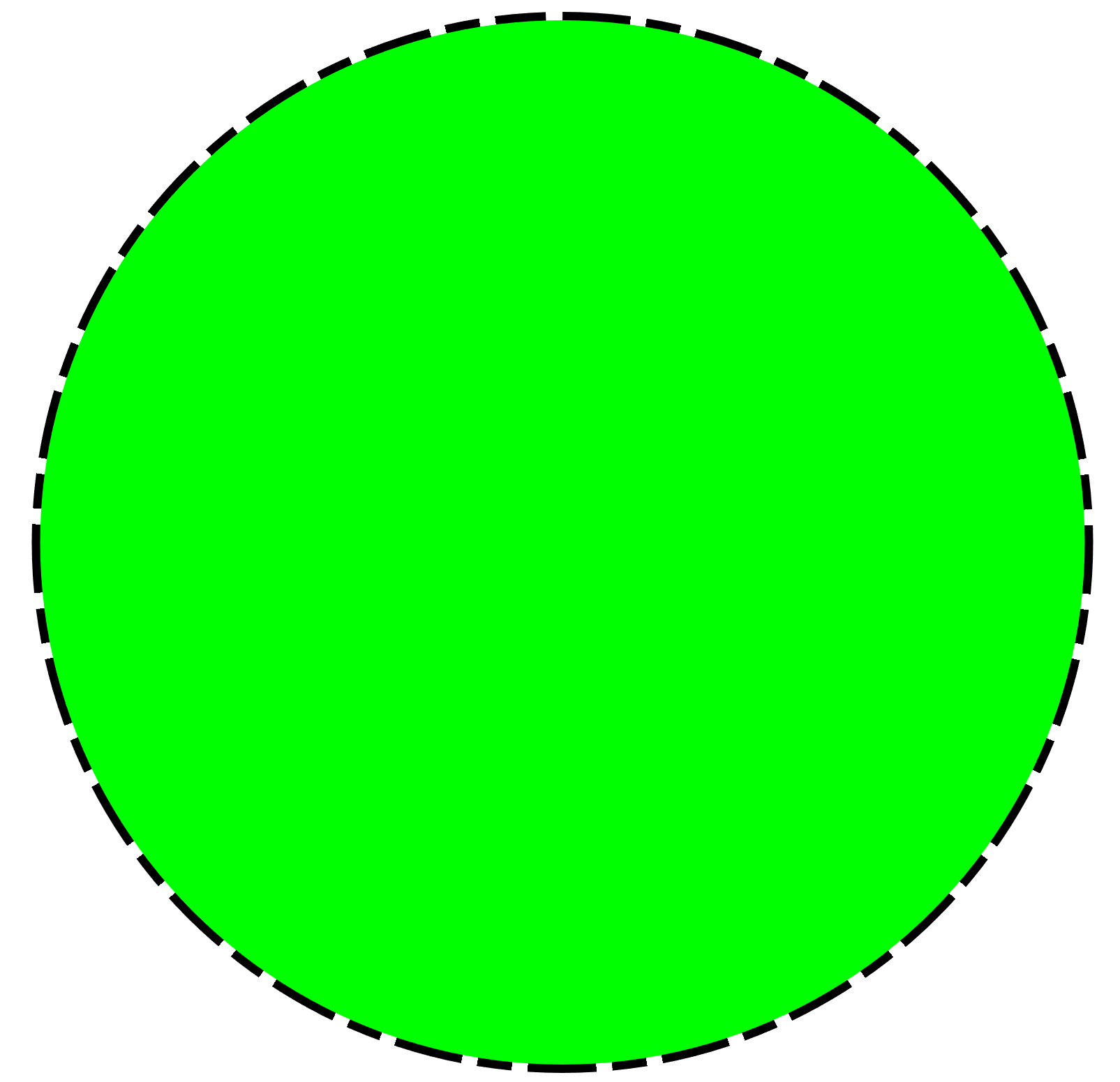} & \overset{X}{\longrightarrow} &
    \includegraphics[scale=0.1]{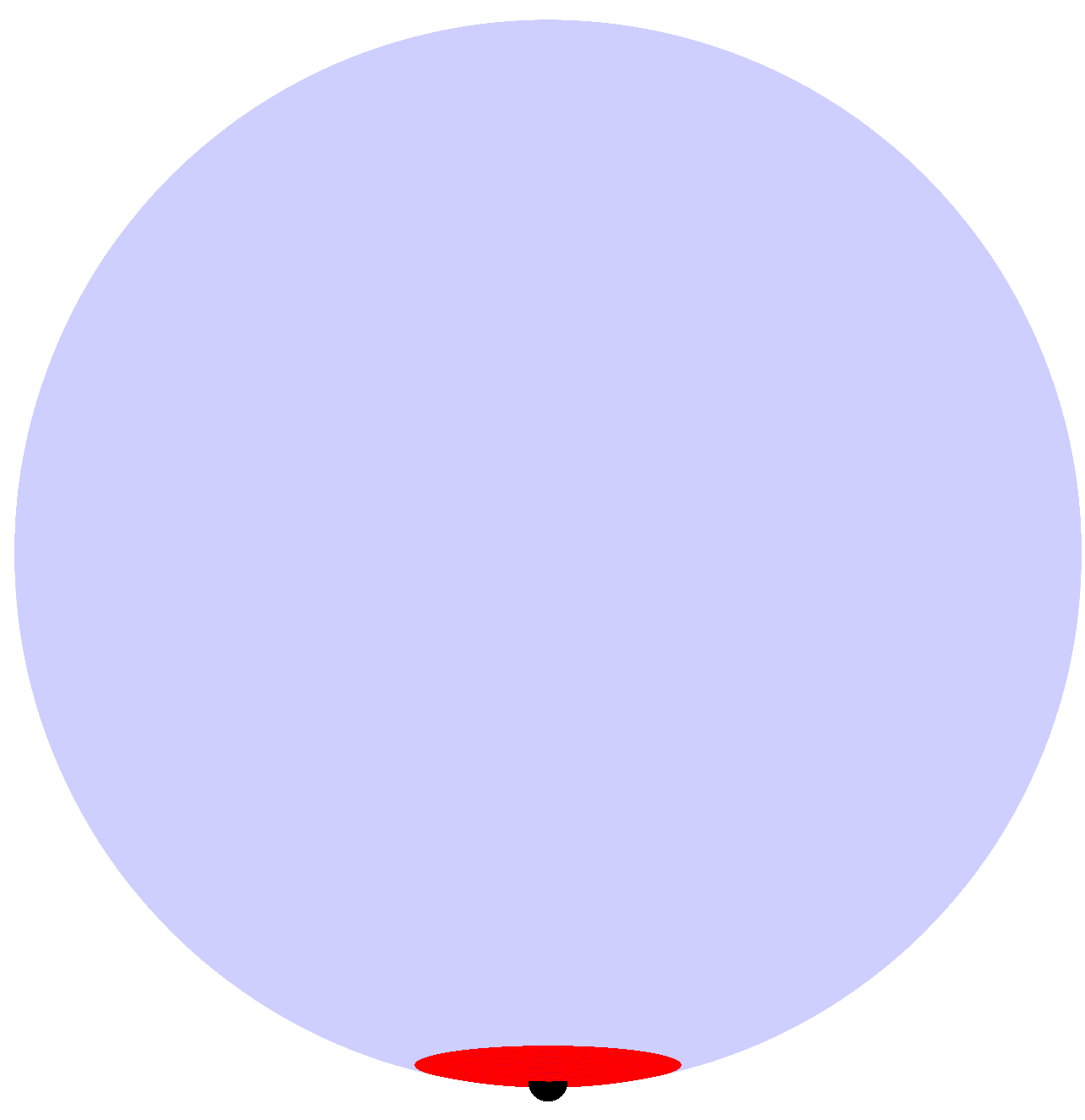}\\
    V & & x^*\in U\cap\Omega 
    \end{array}$
    \caption{Illustrating a closed parametric map at $x^*=(0,0,-1)$
    on the sphere from Ex.~\ref{ex:IllustrativeSimpleSurfaces}}
    \label{Fig:SurfacePlots}
\end{figure}
\end{example}

We now turn to consider \eqref{Eq: ellip_pde} on $\Omega\subset\bR^n$.
Suppose that $G$ is a given metric tensor defined on
the smooth points of~$\Omega$ with inverse $G^{-1}$.
Then, in local coordinates $(t_1,\dots,t_d)$ where $d = \dim \Omega$,
\begin{equation}\label{Eq: LB_local_coord}
\Delta u = \frac{1}{\sqrt{|g|}}\sum_{i=1}^d \frac{\partial}{\partial t_i}\left(\sqrt{|g|}\cdot\sum_{j=1}^d g^{ij} \frac{\partial u}{\partial t_j}\right)
\end{equation}
where $g=\det G$ and $g^{ij}$ is the $(i,j)$ entry of $G^{-1}$.

\begin{example}
For $\Omega = \bR^n$ with
the standard metric tensor $G = I_{n}$, the $n\times n$ identity matrix, 
the local coordinates are simply the standard
coordinates $(x_1,\dots,x_n)$, 
$g = \det G = 1$, and $g^{ij} = \delta_{ij}$ (Kronecker delta).  Hence, 
$$\Delta u = \sum_{i=1}^n \frac{\partial^2 u}{\partial x_i^2}$$ 
which is simply the Laplacian~of $u$~on~$\bR^n$.
\end{example}

\begin{example}
Reconsider the unit circle $\Omega=\{x_1^2+x_2^2=1\}\subset\bR^2$ 
with parameterization
$$X(\theta)=(x_1(\theta),x_2(\theta))=(\cos(\theta),\sin(\theta))
\hbox{~~~~~for~~}\theta\in[0,2\pi]$$
from Ex.~\ref{ex:IllustrativeSimpleCurves}.
Since
$$g=\|X'(\theta))\|^2=
\sin^2(\theta) + \cos^2(\theta) = 1,$$ 
we know that $G = G^{-1} = [1]$.  Hence,
$$\Delta u = \frac{d^2u}{d\theta^2}.$$
For example, if $u(x) = x_1+x_2$, then $u(\theta) = \cos(\theta)+\sin(\theta)$
with
$$\Delta u = \frac{d^2}{d\theta^2}(\cos(\theta)+\sin(\theta)) = -(\cos(\theta)+\sin(\theta)) = -u.$$

If, instead, we utilize the rational parameterization 
$$X(t)=(x_1(t),x_2(t))=\left(\frac{1-t^2}{1+t^2},\frac{2t}{1+t^2}\right) \hbox{~~~~~for~~}t\in\bR,$$
then 
$$g=\|X'(t)\|^2=\left(\frac{-4t}{(1+t^2)^2}\right)^2 + \left(\frac{2(1-t^2)}{(1+t^2)^2}\right)^2 = \frac{4}{(1+t^2)^2}$$
with $G = [g]$ and $G^{-1} = [g^{-1}]$.
Hence,
$$\Delta u = \frac{1+t^2}{2} \frac{d}{dt}\left(\frac{1+t^2}{2} \frac{du}{dt} \right) = \frac{1+t^2}{4}\left((1+t^2)\frac{d^2u}{dt^2}+2t\frac{du}{dt}\right)
= \frac{(1+t^2)^2}{4}\frac{d^2u}{dt^2}+
\frac{t(1+t^2)}{2}\frac{du}{dt}.$$
Similar as above, if $u(x) = x_1+x_2$, then $u(t) = (1+2t-t^2)/(1+t^2)$ and one can verify that
$$\Delta u = -\frac{1+2t-t^2}{1+t^2} = -u.$$
\end{example}

\begin{example}
For the unit sphere $\Omega=\{x_1^2+x_2^2+x_3^2=1\}\subset\bR^3$, consider the parameterization
$$X(\theta_1,\theta_2)=(\sin(\theta_1)\cos(\theta_2),~\sin(\theta_1)\sin(\theta_2),~\cos(\theta_1))
\hbox{~~~~~for~~}\theta_1\in[0,\pi] \hbox{~~and~~} \theta_2\in[0,2\pi].$$  The metric tensor is $$G=\left[\frac{\partial x}{\partial\theta_{i}}\cdot\frac{\partial x}{\partial\theta_{j}}\right]_{i,j}=\left[\begin{array}{cc}
1 & 0 \\0 & \sin^2(\theta_1)
\end{array}\right]
\hbox{~~~with~~~}
G^{-1} = \left[\begin{array}{cc}
1 & 0 \\0 & \csc^2(\theta_1)
\end{array}\right]$$
yielding $g=\det G = \sin^2(\theta_1)$.  Note that since $\theta_1\in[0,\pi]$, 
$\sqrt{|g|}=\sin(\theta_1)\geq0$. 
Therefore, 
$$\begin{array}{rcl}
\Delta u &=&
\displaystyle
\frac{1}{\sin(\theta_1)}\left(\frac{\partial}{\partial \theta_1}
\left(\sin(\theta_1)\frac{\partial u}{\partial \theta_1}\right)
+ \frac{\partial}{\partial \theta_2}\left(\sin(\theta_1) \csc^2(\theta_1) \frac{\partial u}{\partial \theta_2}\right)\right) \\[0.15in]
&=& \displaystyle
\frac{\partial^2 u}{\partial \theta_1^2} + \csc^2(\theta_1)\frac{\partial^2 u}{\partial \theta_2^2} + \cot(\theta_1)\frac{\partial u}{\partial \theta_1}.
\end{array}
$$
For example, if $u(x) = x_1+x_2+x_3$, then $u(\theta) = 
\sin(\theta_1)(\sin(\theta_2)+\cos(\theta_2))+\cos(\theta_1)$
with
$$\Delta u = -u-\frac{\sin(\theta_2)+\cos(\theta_2)}{\sin(\theta_1)}
+ \left(\cos^2(\theta_1)\frac{\sin(\theta_2)+\cos(\theta_2)}{\sin{\theta_1}} - \cos(\theta_1)\right) = -2u.$$
\end{example}

\subsection{Well-posedness for curves}\label{sec:wellposedness}

Let $H^1(\Omega)$ denote the Sobolev space with $k=p=1$ and vanishing boundary set $\Omega$, and $H^{-1}(\Omega)$ denote the dual space to $H^1(\Omega)$.
When $\Omega$ is understood, we simply write $H^1$ and $H^{-1}$, respectively. 

The following provides our main theoretical result about
well-posedness of \eqref{Eq: ellip_pde} for curves.

\begin{theorem}\label{thm:Curves}
If $\Omega$ is a closed parametric $C^1$ curve and $f,c\in H^{-1}$ with $c\geq0$ and $\int_{\Omega}c>0$,
then there exists a unique weak 
solution $u\in H^1$ to \eqref{Eq: ellip_pde}. 
\end{theorem} 

\begin{proof}
We first define a weak solution to~\eqref{Eq: ellip_pde} by multiplying $v\in H^1$ to both sides of~\eqref{Eq: ellip_pde} and applying Green's first identity.  Hence, for the standard inner product $\langle\cdot,\cdot\rangle$, we have
\begin{equation}\label{Eq: weakform}
	\int_{\Omega}(-\Delta u+cu)vdx= \int_{\Omega}\langle\nabla u,\nabla v\rangle dx +\int_{\Omega} cuv dx=\int_\Omega fvdx.
\end{equation}
In particular, a function $u\in H^1$ is called a weak solution to 
\eqref{Eq: ellip_pde} if \eqref{Eq: weakform} is satisfied for all $v\in H^1$.  Consider writing \eqref{Eq: weakform} in the following bilinear
form:
\begin{equation}\label{Eq: weakfunctional}
	a(u,v)=l(v)
\end{equation}
where
\begin{equation}\label{Eq: Bilinear a}
 a(u,v):=\int_{\Omega}\langle\nabla u,\nabla v\rangle dx+\int_\Omega cuvdx
 \hbox{~~~~~and~~~~~} l(v):=\int_\Omega fvdx.\end{equation}
 Then, we can prove \eqref{Eq: ellip_pde} has a unique weak solution in $H^1$ using Lax-Milgram Theorem.

Define
$$\langle \alpha,\beta\rangle_{L^2} := \int_\Omega \langle \alpha,\beta\rangle dx
\hbox{~~~~and~~~~}
\langle u,v\rangle_{\Omega}:=\langle\nabla u,\nabla v\rangle_{L^2}+\langle \sqrt{c}\cdot u,\sqrt{c}\cdot v\rangle_{L^2}.$$ 
The assumptions on $c$ imply that $\langle\cdot,\cdot\rangle_\Omega$ is an
inner product.  In fact, when $c\equiv 1$, $\langle\cdot,\cdot\rangle_\Omega$ 
is the default inner product on $H^1$. 
Let $\|\cdot\|_{\Omega}$ on $H^1$ be the norm induced by  $\langle\cdot,\cdot\rangle_\Omega$. Next, we show the coercivity of
the bilinear function $a(\cdot,\cdot)$. 
To that end, for any $v\in H^1$,
	\begin{equation*}
		a(v,v)=\|\nabla v\|^2_{L^2}+\|\sqrt{c}v\|^2_{L^2}=\|v\|_\Omega^2.
	\end{equation*}
Given $u,v\in H^1$, we square both sides of \eqref{Eq: Bilinear a} and apply the Cauchy-Schwarz inequality to obtain 
	 \begin{equation*}
	 \begin{array}{ll}
	 a(u,v)^2&	\displaystyle{=\langle\nabla u,\nabla v\rangle_{L^2}^2+\left(\int_\Omega cuvdx\right)^2+2\langle\nabla u,\nabla v\rangle_{L^2}\int_\Omega cuvdx}\\[0.15in]
	 	& \displaystyle{\le\|\nabla u\|_{L^2}^2\|\nabla v\|_{L^2}^2+\|\sqrt{c}\cdot u\|_{L^2}^2\|\sqrt{c}\cdot v\|_{L^2}^2+2\|\nabla u\|_{L^2}\|\nabla v\|_{L^2}\|\sqrt{c}\cdot u\|_{L^2}\|\sqrt{c}\cdot v\|_{L^2}}\\ [0.15in]
	 	&	\displaystyle{\le (\|\nabla u\|_{L^2}^2+\|\sqrt{c}\cdot u\|_{L^2}^2)(\|\nabla v\|_{L^2}^2+\|\sqrt{c}\cdot v\|_{L^2}^2)}\\[3pt]
	 	&\le\|u\|_\Omega^2\|v\|_\Omega^2
	 	\end{array}
	 \end{equation*}
	 which shows the boundedness of $a(\cdot,\cdot)$. 
	 Since $f\in H^{-1}$, it follows immediately from the Lax-Milgram Theorem that there exists a unique $u\in H^1$ satisfying \eqref{Eq: weakfunctional}.
\end{proof}

Theorem~\ref{thm:Curves} extends well-posedness 
of \eqref{Eq: ellip_pde} to some curves which have singularities such
as the lemniscate of Gerono shown in Fig.~\ref{Fig:Lemniscate}(b)
for appropriate choices of $f$ and $c$. In particular, Theorem~\ref{thm:Curves} assumes minimum regularity requirement on $f$ and $c$. In the following examples in Sections~\ref{sec:wellposedness} and~\ref{Sec:1d}, $f$ and $c$ have much nicer properties so that a classical solution 
exists, which must be the unique
solution by Theorem~\ref{thm:Curves}.  
By combining these properties
together with Theorem~\ref{thm:Curves}
and applying the Solobev embedding theorem, the solutions to \eqref{Eq: ellip_pde} 
satisfy more regularity conditions 
leading to the 
results in Theorem~\ref{Thm: 1d_2nd_g} below.

\begin{example}\label{ex:LemniscateSolve}
Let $\Lambda$ be the lemniscate of Gerono as in Ex.~\ref{ex:IllustrativeSimpleCurves}.
Consider the linear elliptic PDE
\begin{equation}\label{eq:LemniscatePDE}
-\Delta u + c(x)\cdot u = \pi\cdot x_1
 \qquad \hbox{~on~} \Lambda
\qquad
\hbox{~where~}
c(x) = \pi - \frac{4x_1^2 + 4x_2^2 - 3}{8x_1^2x_2^2 + 16x_2^4 - 3x_1^2 - 17x_2^2 + 4}.
\end{equation}
One can observe that $c\geq0$ and $\int_\Lambda c > 0$
by considering Fig.~\ref{fig:LemniscateC} which
plots $c(X(\theta))$ for $\theta\in[0,2\pi]$
where $X(\theta) = (\cos(\theta),\sin(2\theta)/2)$
is the global parameterization of $\Lambda$
as in Ex.~\ref{ex:IllustrativeSimpleCurves}.
Hence, Thm.~\ref{thm:Curves} shows that there exists a unique solution to \eqref{eq:LemniscatePDE}.
In fact, using~\eqref{Eq: LB_local_coord}, it is easy to verify that $u(x) = x_1$ 
solves~\eqref{eq:LemniscatePDE}.  This problem 
will be reconsidered 
numerically in Ex.~\ref{ex:LemniscateSolveNum}.

\begin{figure}[!ht]
    \centering
    \includegraphics[scale=0.2]{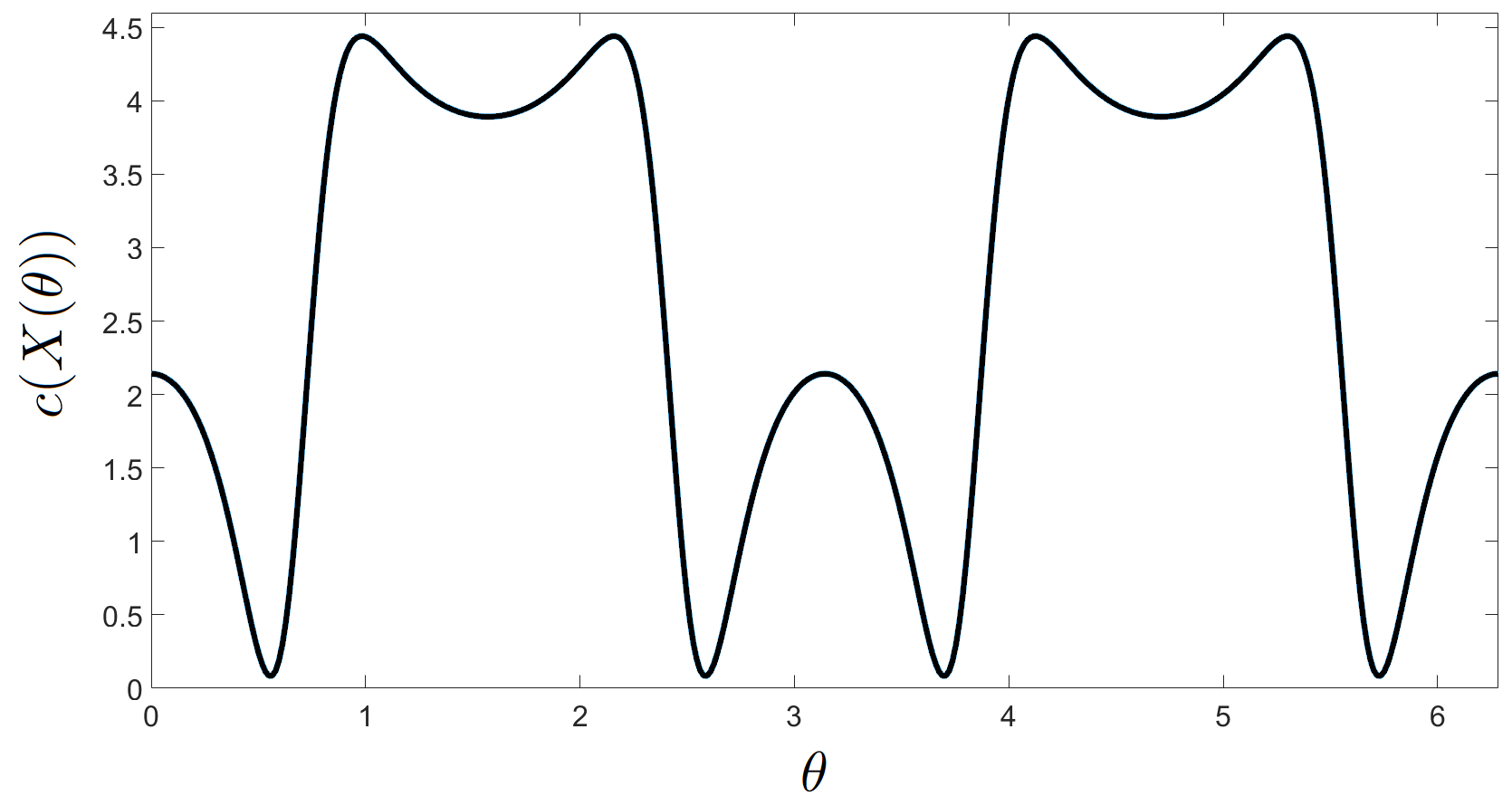}
    \caption{Plot of $c(X(\theta))$ with respect to $\theta\in[0,2\pi]$
    from Ex.~\ref{ex:LemniscateSolve}.
    }
    \label{fig:LemniscateC}
\end{figure}
\end{example}

Building on the existence
and uniqueness result provided by Theorem~\ref{thm:Curves}, the following
develops approaches for 
numerically computing the solution to \eqref{Eq: ellip_pde}
when a globabl parameterization is known.

\subsection{Solving with a global parameterization}\label{sec:global_param}
 
 When the real algebraic curve $\Omega\subset\bR^n$ 
 is a closed parametric $C^1$
 curve with a given parameterization 
 $X:[a,b]\mapsto\Omega$
 such that $X'(t)\neq 0$ for all $t\in[a,b]$,
 solving~\eqref{Eq: ellip_pde} reduces to solving
 an ordindary differential equation 
 on $[a,b]$ with periodic boundary as follows.  
 By definition, $g(t) = \|X'(t)\|^2 > 0$,
 $G(t) = [g(t)]$, and $G^{-1}(t) = [g^{-1}(t)]$
 for $t\in[a,b]$.
With \eqref{Eq: LB_local_coord}, 
the linear elliptic PDE
\eqref{Eq: ellip_pde} simplifies~to
\begin{equation}\label{Eq: ODE_1d}
	\begin{array}{l}
\displaystyle{-\frac{1}{g}}\frac{d^2u}{dt^2}+\frac{1}{2g^2}\frac{dg}{dt}\frac{du}{dt}+c\cdot u=f \qquad\hbox{~on~} [a,b]\\
\end{array}
\end{equation}
with periodic boundary
where, by abuse of notation, $c$ and $f$ are 
the corresponding restrictions.
Therefore,
one can, for example, simply use a finite difference approach
with a three-point stencil to discretize~\eqref{Eq: ODE_1d} as follows.
Given $N$, consider $\Delta t = (b-a)/N$
with \mbox{$t_i = a + i\cdot \Delta t$} for $i = 0,\dots,N$.
Since $u$ is periodic on $[a,b]$ with
$t_0=a$ and $t_N=b$, 
we aim to compute~$u_i$ for $i=0,\dots,N-1$ such that
$u_i \approx u(t_i)$ which amounts to computing
\mbox{$U_N = (u_0,\dots,u_{N-1})^T$} that solves
the linear system
\begin{equation}\label{Eq:1d_disc_g}
    A_N\cdot U_N = F_N
\end{equation}
where $F_N = (f(t_0),\dots,f(t_{N-1}))^T$ and
\begin{equation}\label{eq:AN}
A_N = \left(\begin{array}{cccccc}
 	C_0&R_0&0&\cdots&0&L_0\\
 	L_1&C_1&R_1&0&\cdots&0\\
 	0&L_2&C_1&R_2&\cdots&0\\
 	\vdots&&\ddots&\ddots&\ddots&\vdots\\
 	0&\cdots&0&L_{N-2}&C_{N-2}&R_{N-2}\\
 	R_{N-1}&0&\cdots&0&L_{N-1}&C_{N-1}
 	
 	\end{array}
 	\right)
\end{equation}
such that
$$
L_i=-\frac{1}{g(t_i)\Delta t^2}\left(1+\frac{g^{\prime}(t_i)\Delta t}{4g(t_i)}\right),
~~
  C_i=c(t_i)+\frac{2}{g(t_i)\Delta t^2},
~~\hbox{~and~}~~
  R_i=-\frac{1}{g(t_i)\Delta t^2}\left(1-\frac{g^{\prime}(t_i)\Delta t}{4g(t_i)}\right).$$  
By imposing a stronger condition
on the regularity of the solution $u$ to~\eqref{Eq: ellip_pde}, 
namely $u\in C^4(\Omega)\subset H^1$,
we obtain the following.

 \begin{theorem}\label{Thm: 1d_2nd_g}
If $u\in C^4(\Omega)$ and there exists $\delta >0$
such that $c>\delta$, then the numerical scheme \eqref{Eq:1d_disc_g} is convergent, stable,  	 and has second order accuracy.
 \end{theorem}
\begin{proof}
Using Taylor series expansion, we have
	\begin{equation}\label{Eq:Taylor_uniform}
		\begin{array}{cc}
		u(x(t_{i+1}))=u(x(t_i))+\Delta t u^{\prime}(x(t_i))+\frac{\Delta t^2}{2}u^{\prime\prime}(x(t_i))+\frac{\Delta t^3}{3!}u^{\prime\prime\prime}(x(t_i))+\frac{\Delta t^4}{4!}u^{\prime\prime\prime\prime}(x(\eta_i)),\\
		u(x(t_{i-1}))=u(x(t_i))-\Delta t u^{\prime}(x(t_i))+\frac{\Delta t^2}{2}u^{\prime\prime}(x(t_i))-\frac{\Delta t^3}{3!}u^{\prime\prime\prime}(x(t_i))+\frac{\Delta t^4}{4!}u^{\prime\prime\prime\prime}(x(\xi_i)),
		\end{array}
	\end{equation}
	where $\eta_i\in [t_i,t_{i+1}]$ and $\xi_i\in[t_{i-1},t_i]$. Therefore,
	\begin{equation*}
		\dfrac{u(x(t_{i+1}))-2u(x(t_i))+u(x(t_{i-1}))}{\Delta t^2}=u^{\prime\prime}(x(t_i))-\frac{\Delta t^2}{4!}(u^{\prime\prime\prime\prime}(x(\eta_i))+u^{\prime\prime\prime\prime}(x(\xi_i))).
	\end{equation*}
	This expression combined with \eqref{Eq:1d_disc_g} yields \begin{equation}\label{Eq:1d_exact_sch_g}
		A_N\left(\begin{array}{c}
		u(x(t_0))\\ \vdots\\u(x(t_{N-1}))
		\end{array}\right)+\frac{\Delta t^2}{4!}\left(\begin{array}{c}
		u^{\prime\prime\prime\prime}(x(\eta_0)) {+ u^{\prime\prime\prime\prime}(x(\xi_0))}\\ \vdots\\u^{\prime\prime\prime\prime}(x(\eta_{N-1})) {+ u^{\prime\prime\prime\prime}(x(\xi_{N-1}))}
		\end{array}\right)=F_N.
	\end{equation}
	Denoting
	\[
	u_N=\left(\begin{array}{c}
	u(x(t_0))\\ \vdots\\u(x(t_{N-1}))
	\end{array}\right) \hbox{~~~~and~~~~} u^{\prime\prime\prime\prime}_N= \left(\begin{array}{c}
	u^{\prime\prime\prime\prime}(x(\eta_0)) {+ u^{\prime\prime\prime\prime}(x(\xi_0))}\\ \vdots\\u^{\prime\prime\prime\prime}(x(\eta_{N-1})) {+ u^{\prime\prime\prime\prime}(x(\xi_{N-1}))}
	\end{array}\right),
	\] 
	subtracting \eqref{Eq:1d_disc_g} from \eqref{Eq:1d_exact_sch_g} yields
	\begin{equation*}
		A_N(u_N-U_N)=-\frac{\Delta t^2}{4!}u^{\prime\prime\prime\prime}_N.
	\end{equation*}
	Thus, the error satisfies
	\begin{equation}\label{Eq:1d_bound_gerr}
		\|u_N-U_N\|_{\infty}=\frac{\Delta t^2}{4!}\|A_N^{-1}u^{\prime\prime\prime\prime}_N\|_{\infty}\le \frac{\Delta t^2}{4!}\|A_N^{-1}\|_{\infty}\|u^{\prime\prime\prime\prime}_N\|_\infty.
	\end{equation}
    For sufficiently small $\Delta t$, 
	one can assume that $C_i>\delta + \frac{2}{g(t_i)\Delta t^2}>0$ while 
    both $|L_i|$ and $|R_i|$ are bounded above by, say,
    $\frac{\delta}{4} + \frac{1}{g(t_i)\Delta t^2}$.
    Thus, we have
	\begin{equation*}
		C_i-(|L_i|+|R_i|)>\frac{\delta}{2} > 0.
	\end{equation*}
    Hence, $A_N$ is a strictly diagonally dominant matrix 
    so that $A_N$ is invertible where the real parts
    of the eigenvalues are positive so the stability of the scheme
    follows immediately.
    Moreover, the Ahlberg-Nilson-Varah bound \cite{Varah1975,AN1963} yields
    $\|A_N^{-1}\|_{\infty} \leq \frac{2}{\delta} < \infty$
	showing that 
	\begin{equation*}
		\|u_N-U_N\|_{\infty} \le \frac{2\cdot \Delta t^2}{\delta\cdot 4!}\|u^{\prime\prime\prime\prime}_N\|_\infty.
	\end{equation*}
	Since $u\in C^4(\Omega)$, the global error defined above for scheme \eqref{Eq:1d_disc_g} is bounded and converges to 0 as the mesh size goes to zero. In particular, the scheme is convergent with second order accuracy.
\end{proof}
 \everymath{}

Of course, one can repeat this construction using
a larger stencil and imposing a stronger condition
on the regularity of the solution to obtain
higher order accuracy.  The following illustrates
the convergence rate for the three-point stencil
using a five-point stencil with many points 
to estimate the error.

\begin{example}\label{ex:Ellipse_comparison1}
Consider solving 
\begin{equation}\label{eq:Ellipse_comparison1}
	-\Delta u + u = x \qquad \hbox{~on~~~} x^2 + 50y^2 = 1.
\end{equation} 
Using the global parameterization
\begin{equation}\label{eq:EllipseGlobalParam}
X(\theta) = \left(\sin\theta,\frac{\cos\theta}{\sqrt{50}}\right),
~~~~~~~\theta\in[0,2\pi],
\end{equation}
one aims to solve
$$-\frac{50}{50-49\sin^2\theta} u_{\theta\theta} -
\frac{2450\sin\theta\cos\theta}{(50-49\sin^2\theta)^2}
u_\theta + u = \sin\theta \qquad\hbox{~on~} [0,2\pi]$$
such that $u$ is periodic on $[0,2\pi]$.
 Table~\ref{tab:ellipse_comparison1} lists
 the error and convergence order
which computationally verifies second order
convergence as expected by Thm.~\ref{Thm: 1d_2nd_g}.
Here, the error is computed by comparing against
the solution obtained using a five-point stencil
with $N=\mbox{20,480}$.

\begin{table}[h!]
\begin{center}
\begin{tabular}{P{20mm} | P{20mm} | P{20mm} }
    $N$ & $L_\infty$ Error & Order  \\
    \hline \hline
    %\cline{2-3}
    160&2.043$\cdot 10^{-4}$&---\\
    320&5.099$\cdot 10^{-5}$&2.002\\
    640&1.274$\cdot 10^{-5}$&2.000\\
    1280&3.185$\cdot 10^{-6}$&2.000\\
\end{tabular}
\end{center}
\caption{Comparison of error for solving \eqref{eq:Ellipse_comparison1} using
the global parameterization \eqref{eq:EllipseGlobalParam}.}
\label{tab:ellipse_comparison1}
\end{table}
\end{example}

\section{Local parameterization for curves}\label{Sec:1d}

When there is no readily available
global parameterization, one can
solve \eqref{Eq: ellip_pde}
via a finite difference method
based on local parameterization at each
sample point.  
The following proceeds
by first considering a numerical cell decomposition using
numerical algebraic geometry, then analyzing a local
tangential parameterization at smooth points,
and finally considering singular points.

\subsection{Curve decomposition using numerical algebraic geometry}\label{sec:DecompNAG}

One approach for decomposing a curve is to utilize
a numerical cellular decomposition \cite{BertiniReal,RealCurveDecomp}
computed using numerical algebraic geometry \cite{BHSW:BertiniBook,SW05}.
A cellular decomposition of a curve is 
a disjoint union of finitely many vertices $V$,
which are simply points on the curve,
and edges $E$, which are portions of the curve diffeomorphic to an interval in $\bR$.  The endpoints of each edge are vertices.  In particular, $V$ must contain the 
set of singular points of the curve.

\begin{example}\label{ex:LemniscateCellDecomp}
Reconsider the lemniscate of Gerono $\Lambda\subset\bR^2$
defined in Ex.~\ref{ex:IllustrativeSimpleCurves}
and shown in Fig.~\ref{fig:IntroPlot}(b).
Figure~\ref{fig:LemniscateCellDecomp} illustrates
a cellular decomposition of $\Lambda$ consisting
of 3 vertices and~4~edges. 

\begin{figure}[!ht]
    \centering
    \includegraphics[scale=0.6]{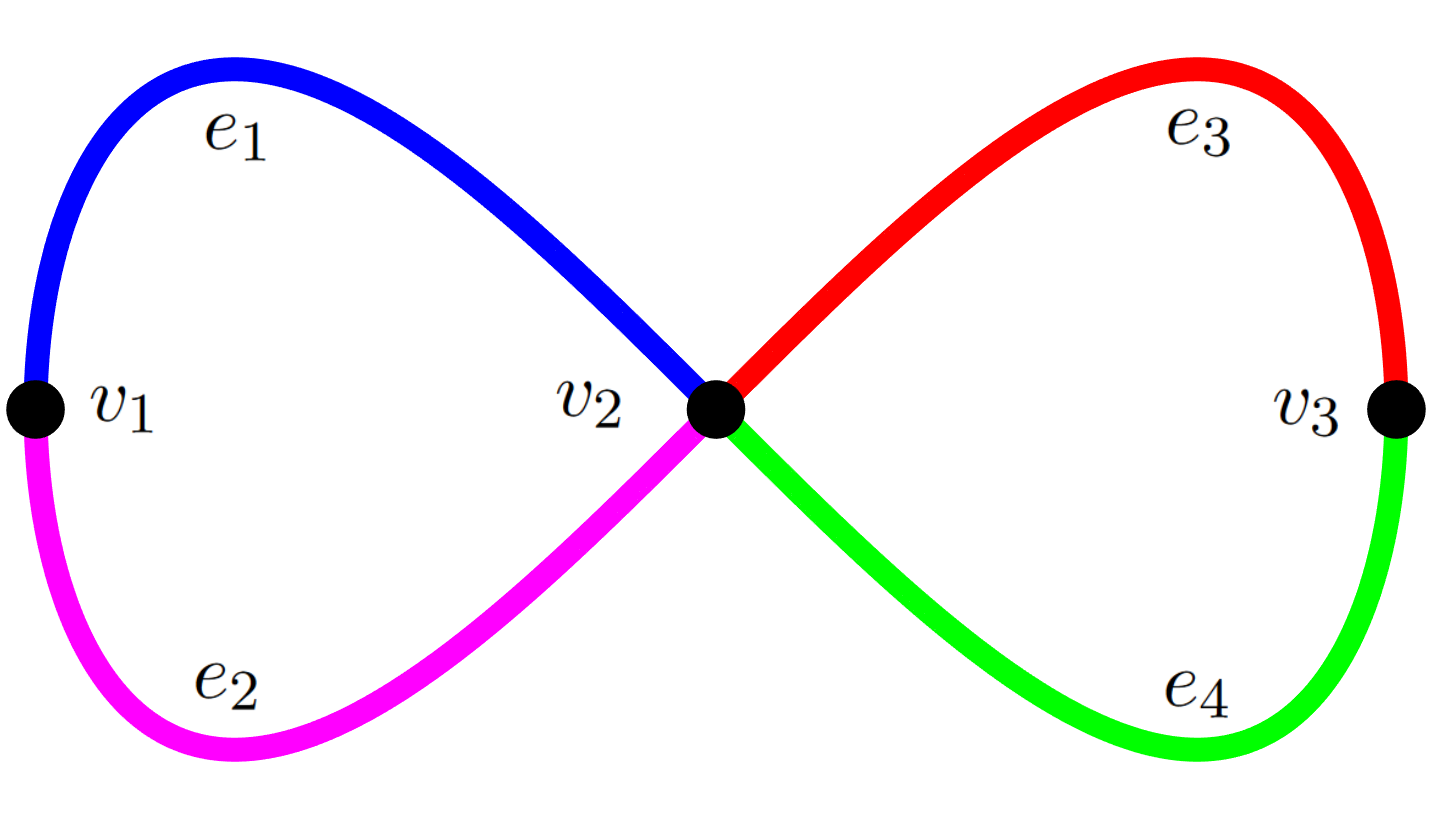}
    \caption{Cellular decomposition for lemniscate of Gerono
    with vertices $v_i$ and edges $e_j$}
    \label{fig:LemniscateCellDecomp}
\end{figure}
\end{example}

A numerical cellular decomposition simply represents
each edge of a cellular decomposition 
by an interior point along with a homotopy
that permits the tracking along the edge starting
from the interior point.  From this numerical representation,
one can perform computations on each edge.  
For example, one can sample points along each edge
and construct a Chebyshev interpolant as described in \cite{BertiniCheby}.  From the Chebyshev interpolant,
one can easily approximate the arc length 
of each edge and approximate mesh points in the desired
structure, for example, uniform in arc~length.

At each point on the curve, there is a local irreducible
decomposition of the curve at the point 
which can be computed using numerical algebraic geometry
\cite{LocalNID}.  
A curve is locally irreducible at every smooth point
on the curve and is locally diffeomorphic to the tangent line.
This is utilized next to construct a tangential parameterization at smooth points.
The only points on a curve where the curve could be
locally reducible is at a singular point.
Hence, at each singular point on the curve,
the approach in \cite{LocalNID} uses the 
local monodromy group structure computed 
using a homotopy to determine the locally irreducible components of the curve at a singular point.  
Moreover, each locally irreducible component has a 
well-defined local degree \cite{LocalNID}.  
If a component has local degree equal to $1$, 
then it is locally diffeomorphic to a tangent line.

\begin{example}
Continuing from Ex.~\ref{ex:LemniscateCellDecomp},
all points are smooth points of $\Lambda$ except
$v_2=(0,0)$.  At~$v_2$,~$\Lambda$ decomposes
into two locally irreducible components
each of local degree $1$ corresponding to each 
of the two local tangent directions at $v_2$.
\end{example}

Local irreducible decomposition is important for
solving~\eqref{Eq: ellip_pde} since 
Theorem~\ref{thm:Curves} enforces that 
the solution is continuous along each locally irreducible
component.  Hence, a numerical solving scheme
needs to allow for a singular point to take a different
value along each locally irreducible component passing
through the singular point
as illustrated in Figure~\ref{fig:IntroPlot}(b).

\subsection{Tangential parameterization at smooth points}\label{sec:localTangent}
 
 The following uses an approach based on a local tangential parameterization for a smooth curve 
 to compute $x'(s)$ and obtain $g$ which greatly simplifies the calculation of coefficients for the numerical scheme. 
 Let $\pi_N = \{p_0,p_1,\dots,p_{N-1}\}$ consist
 of $N$ mesh points uniformly distributed in arc length
using a cyclic ordering with $p_i = p_{N+i}$ as needed.
Define $[p_{i-1},p_{i+1}]$
 to be the segment of the curve passing through
 points $p_{i-1}$, $p_i$, and $p_{i+1}$.  
Let $v_i$ be a unit tangent vector to the curve
at $p_i$ and consider $\ell_i(t) = p_i + t v_i$
which parameterizes the tangent line to the curve
at $p_i$.
 Consider the map $\alpha_i:[p_{i-1},p_{i+1}]\rightarrow \bR$
 defined by $\alpha_i(p) = (p-p_i)\cdot v_i$.
 By replacing $v_i$ by $-v_i$ as needed and taking~$N$
 large enough, $\alpha_i$ is a diffeomorphism from
$[p_{i-1},p_{i+1}]$ to $[\alpha_i(p_{i-1}),\alpha_i(p_{i+1})]$
where 
$$\alpha_i(p_{i-1}) < 0 = \alpha_i(p_i) < \alpha_i(p_{i+1}).$$
See Fig.~\ref{fig:1d-tangentialparam} for an illustration of this
tangential parameterization construction.
  
\begin{figure}[!hb]
    \centering
    \includegraphics[scale = 0.3]{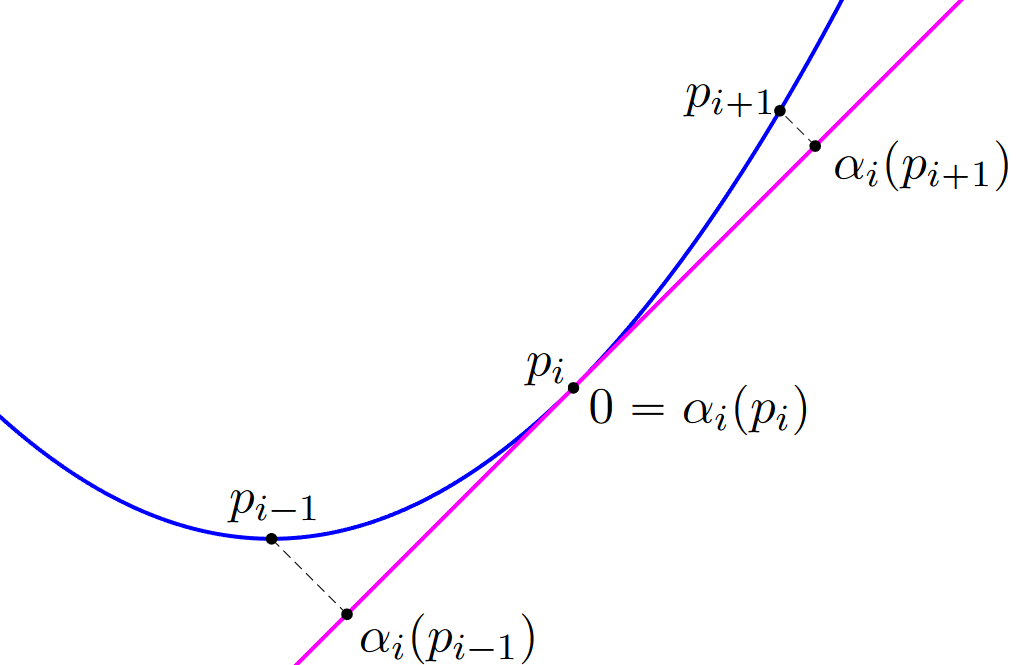}
    \caption{Illustration of tangential parameterization. }
    \label{fig:1d-tangentialparam}
\end{figure}

Let $X_i:[\alpha_i(p_{i-1}),\alpha_i(p_{i+1})]\rightarrow[p_{i-1},p_{i+1}]$ be the inverse of $\alpha_i$.  Locally, \eqref{Eq: ODE_1d} using $X_i(t)$
is simplified at $t = 0$ based on the following.

\begin{theorem}\label{Lem: 1d_g=1}
With the setup described above, $X_i'(0) = v_i$.  
Moreover, for corresponding metric tensor $G(t)$, 
\eqref{Eq: LB_local_coord} becomes
$\Delta u(0) = \frac{\partial^2 u(0)}{\partial t^2}$.
\end{theorem}
\begin{proof}
For $t\in[\alpha_i(p_{i-1}),\alpha_i(p_{i+1})]$, 
one knows that $X_i(t)$ satisfies
$$\left[\begin{array}{c} 
(X_i(t) - p_i)\cdot v_i - t \\
F(X_i(t)) \end{array}\right] = 0.
$$
By the implicit function theorem,
\begin{equation}\label{Eq: Jacob}
X_i'(t) = -\left[\begin{array}{c}
v_i^T \\ 
JF(X_i(t)) \end{array}\right]^{-1}
\left[\begin{array}{c} -1 \\ 0 \end{array}\right].
\end{equation}
Since $v_i\cdot v_i = 1$ and $JF(X_i(0)) v_i = 0$, it immediately
follows from \eqref{Eq: Jacob} that $X_i'(0) = v_i$.
With $g(t) = G(t) = \|X_i'(t)\|^2$, 
$g_i(0) = 1$.  Additionally, 
from the first row of \eqref{Eq: Jacob}, we know $v_i\cdot X_i'(t) = 1$
so that $v_i\cdot X_i''(t) = 0$.  Hence, at $t = 0$, $X_i'(0)\cdot X_i''(0) = 
v_i \cdot X_i''(0) = 0$ which immediately yields that 
$\frac{dg_i(0)}{dt} = 0$ and the result follows.
\end{proof}

\begin{example}\label{ex:EllipseLocal}
To illustrate, consider the ellipse
$x^2+10y^2=1$ at
$$p=\left[\begin{array}{c} 1 \\ 0 \end{array}\right]
\hbox{~~with~~}
v=\left[\begin{array}{c} 0 \\ 1 \end{array}\right]
\hbox{~~so that~~}
X(t) = \left[\begin{array}{c}
\sqrt{1-10t^2} \\ t
\end{array}\right].
$$
Clearly, $X(0) = p$.  Since 
$$X'(t) = \left[\begin{array}{c} -10t/\sqrt{1-10t^2}
\\ 1
\end{array}\right],$$
it is clear that $X'(0) = v$.
Moreover, for $t$ near $0$, 
\eqref{Eq: LB_local_coord} becomes
$$
\Delta u(t) =
\sqrt{\frac{1-10t^2}{1+90t^2}}\cdot \frac{d}{dt}
\left(
\sqrt{\frac{1-10t^2}{1+90t^2}}\frac{du(t)}{dt}
\right)
=
\frac{1-10t^2}{1+90t^2}\frac{d^2u(t)}{dt}
-\frac{100t}{(1+90t^2)^2}\frac{du(t)}{dt}
$$
which yields $\Delta u(0) = \frac{d^2 u(0)}{dt^2}$
in accordance with Thm.~\ref{Lem: 1d_g=1}.
\end{example}

Combining with \eqref{Eq: ODE_1d}, one can develop a local
discretization to approximate $u(p_i)$ for each $i$
which is simplified due to Theorem~\ref{Lem: 1d_g=1}.  
For example, with $u_i\approx u(p_i)$, 
a three-point stencil yields the following discretization:
\begin{equation}\label{Eq: 1-D discrete_system}
L_i\cdot u_{i-1} + C_i\cdot u_i + R_i\cdot u_{i+1} = f(p_i)
\end{equation}
where
$$
\mbox{\small $
L_i = \dfrac{-2}{\alpha_i(p_{i-1})(\alpha_i(p_{i-1}) - \alpha_i(p_{i+1}))}, ~~
C_i = c(p_i) + \dfrac{-2}{\alpha_i(p_{i+1})\alpha_i(p_{i-1})}, ~~
R_i = \dfrac{-2}{\alpha_i(p_{i+1})(\alpha_i(p_{i+1}) - \alpha_i(p_{i-1}))}.
$}
$$
Writing $U_N = (u_0,\dots,u_{N-1})^T$, and $F_N = (f(p_0),\dots,f(p_{N-1}))^T$, \eqref{Eq: 1-D discrete_system} 
yields the linear system
$$B_N\cdot U_N = F_N$$
where $B_N$ is the same as $A_N$ in \eqref{eq:AN} 
with the localized versions of $L_i$, $C_i$, and $R_i$ above.
In particular, note that this does not require computing
$g_i$.

\begin{theorem}\label{thm:1Dlocal}
If $u\in C^3(\Omega)$ and there exists $\delta >0$
such that $c>\delta$, the finite difference scheme arising
from \eqref{Eq: 1-D discrete_system} is convergent and at least first order accurate in arc length mesh size.  
\end{theorem}
\begin{proof}
The proof is similar to that of Theorem~\ref{Thm: 1d_2nd_g} 
except that \eqref{Eq: 1-D discrete_system} uses an 
unstructured three-point stencil to approximate $\Delta u(p_i)$, 
which becomes the second-order
central difference scheme when $\alpha_i(p_{i+1}) = -\alpha_i(p_i)$.
\end{proof}

\begin{remark}\label{Remark:HigherOrderLocal}
By imposing a stronger condition on the regularity of the solution $u$
as well as increasing the size of the domain $\alpha_i$ 
for which each remains a diffeomorphism, one can naturually
replace the three-point stencil used in \eqref{Eq: 1-D discrete_system} 
with larger stencils and obtain similar results to 
Theorem~\ref{thm:1Dlocal} with higher-order convergence.    
\end{remark}

\begin{example}\label{ex:Ellipse_comparison}
For $a>0$, consider solving
\begin{equation}\label{eq:Ellipse_comparison}
	-\Delta u + u = x \qquad \hbox{~on~} x^2 + ay^2 = 1.
\end{equation} 
Using the global parameterization $\left(\cos t,\dfrac{\sin t}{\sqrt{a}}\right)$ of the ellipse, we first compare 
the global method in Section~\ref{sec:global_param} 
with the local tangential parameterization.
Table~\ref{tab:ellipse_comparison} compares using a three-point
stencil for both when $a=50$ where the errors are computed
by comparing against an approximate solution 
computed using a five-point stencil with the global parameterization
using 20,480 points.

\begin{table}[h!]
\begin{center}
\begin{tabular}{P{20mm} | P{20mm} | P{20mm} | P{20mm} | P{20mm} }
    \hhline{~|====}
     \multicolumn{1}{P{20mm}|}{} & \multicolumn{2}{P{40mm}|}{Global parameterization} & \multicolumn{2}{P{40mm}}{Local tangential parameteization} \\
    \hhline{=====}
    %\multicolumn{3}{P{20mm}}{N}&\\
    $N$ & $L_\infty$ Error & Order & $L_\infty$ Error & Order  \\
    \hline \hline
    %\cline{2-3}
    160&2.043$\cdot 10^{-4}$&---&2.459$\cdot 10^{-3}$&---\\
    320&5.099$\cdot 10^{-5}$&2.002&6.401$\cdot 10^{-4}$&1.942\\
    640&1.274$\cdot 10^{-5}$&2.000&1.630$\cdot 10^{-4}$&1.975\\
    1280&3.185$\cdot 10^{-6}$&2.000&4.094$\cdot 10^{-5}$&1.992\\
\end{tabular}
\end{center}
\caption{Comparison of global and local parameterization methods
for solving \eqref{eq:Ellipse_comparison} when $a=50$.}
\label{tab:ellipse_comparison}
\end{table}

We next compare using a three-point stencil and a five-point stencil
with the local tangential parameterization for
$a=1$, $a=10$, and $a=50$.
The results are summarized in Table~\ref{tab:ellipse}
with the error computed as above.
This shows
that the error decreases when curvature is more uniform throughout
the curve so that the unstructured stencil approaches
a uniformly-spaced stencil.  
Figure~\ref{Fig:Ellipse} shows the numerical solutions 
of \eqref{eq:Ellipse_comparison} for these three instances
using $N = 160$~points.  

\begin{table}[h!]
\begin{center}
\begin{tabular}{P{15mm} | P{20mm} | P{20mm} | P{20mm} | P{20mm} | P{20mm} }
    \hhline{~~|====}
     \multicolumn{2}{P{35mm}|}{} & \multicolumn{2}{P{40mm}|}{3-Point Stencil} & \multicolumn{2}{P{40mm}}{5-Point Stencil} \\
    \hhline{~|=====}
    %\multicolumn{3}{P{20mm}}{N}&\\
    & $N$ & $L_\infty$ Error & Order & $L_\infty$ Error & Order  \\
    \hline \hline
    %\cline{2-3}
    \multirow{4}{*}{$a = 1$}&160&9.639$\cdot 10^{-5}$&---&2.979$\cdot 10^{-7}$&---\\
    &320&2.410$\cdot 10^{-5}$&2.000&1.859$\cdot 10^{-8}$&4.002\\
    &640&6.024$\cdot 10^{-6}$&2.000&1.162$\cdot 10^{-9}$&4.000\\
    &1280&1.506$\cdot 10^{-6}$&2.000&7.322$\cdot 10^{-11}$&3.988\\
    \hline
    \multirow{4}{*}{$a = 10$}&160&5.744$\cdot 10^{-4}$&---&5.191$\cdot 10^{-5}$&---\\
    &320&1.442$\cdot 10^{-4}$&1.995&3.265$\cdot 10^{-6}$&3.991\\
    &640&3.607$\cdot 10^{-5}$&1.999&2.046$\cdot 10^{-7}$&3.997\\
    &1280&9.020$\cdot 10^{-6}$&2.000&1.281$\cdot 10^{-8}$&3.998\\
    \hline
    \multirow{4}{*}{$a = 50$}&160&2.459$\cdot 10^{-3}$&---&7.394$\cdot 10^{-3}$&---\\
    &320&6.401$\cdot 10^{-4}$&1.942&3.119$\cdot 10^{-4}$&4.567\\
    &640&1.630$\cdot 10^{-4}$&1.975&1.968$\cdot 10^{-5}$&3.986\\
    &1280&4.094$\cdot 10^{-5}$&1.992&1.245$\cdot 10^{-6}$&3.983\\
\end{tabular}
\end{center}
\caption{Comparison of using the local 
tangential parameterization method
for different stencil sizes and varying values of $a$
when solving \eqref{eq:Ellipse_comparison}.}
\label{tab:ellipse}
\end{table}

\begin{figure}[h!]
	\centering
	\includegraphics[scale = 0.25]{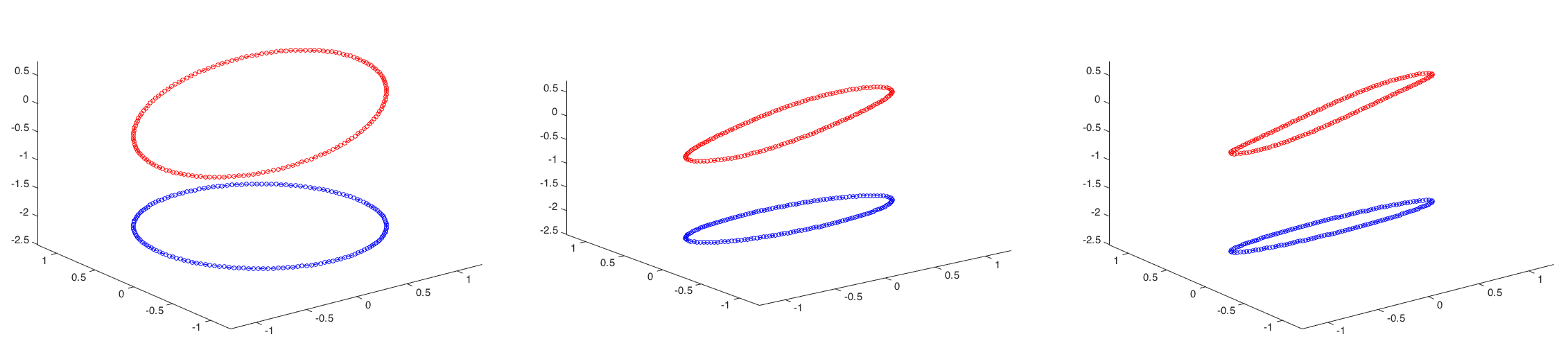} \\
	(a) \hspace{1.77in} (b)\hspace{1.77in}  (c)
	\caption{Solution (red) for $-\Delta u + u = x$ on $x^2 + ay^2 = 1$ (blue) with $N = 160$ mesh points for (a) $a = 1$, (b) $a = 10$, and (c) $a = 50$.}\label{Fig:Ellipse}
\end{figure}
\end{example}

\begin{remark}
Using Theorem~\ref{Lem: 1d_g=1}, this local tangential approximation 
does not encounter the cost of approximating metric tensor coefficients.
Moreover, by using numerical algebraic geometry to perform
computations on the curve $\Omega\subset\bR^n$, we note that we are solving
in the space of $H^1(\Omega)$ instead of the higher-dimensional space $H^1(\bR^n)$. This becomes especially useful for large $n$.
\end{remark}

\subsection{Local parameterization near singularities}

For a smooth curve, every point has a well-defined tangent
direction and the curve has a local tangential parameterization
as illustrated in Figure~\ref{fig:1d-tangentialparam}.
For a singular point, one needs to look at each local irreducible
component and allow the value of $u$ at the singular point to
take a different value along each such component 
as described in Section~\ref{sec:DecompNAG}.  
If a local irreducible component has local degree 1,
it is locally diffeomorphic to a well-defined tangent line so that
the singular point is a smooth point with respect to the local
irreducible component.
Hence, one can
simply apply the local tangential parameterization
from Section~\ref{sec:localTangent} along the local
irreducible component.  

\begin{example}\label{ex:LemniscateSolveNum}
Consider the following problem
\begin{equation}\label{Eq:Lemniscate}
	-\Delta u + u = f(x,y) \qquad \hbox{~on~} x^4 - x^2 + y^2 = 0
\end{equation}  
\noindent
where $f(x,y) = x^2+xy-1$ whose solution was shown in
Fig.~\ref{fig:IntroPlot}(b).  The origin is the only singular
point on the lemniscate of Gerono which arises as the intersection
of two locally irreducible components of local degree 1
so that one can employ a local tangential parameterization
along each locally irreducible component.
Table~\ref{tab:Lemniscate} summarizes the results when
using a local tangential parameterization 
with a three-point stencil 
where the errors are computed using a three-point
stencil with the global parameterization from Ex.~\ref{ex:IllustrativeSimpleCurves} with
$N = \mbox{20,480}$ points.
Figure~\ref{Fig:Figure8} shows two views of the solution
computed using $N = 160$ points.

\begin{table}[h!]
\begin{center}
\begin{tabular}{P{20mm} | P{20mm} | P{20mm}}
    \hline \hline
    %\multicolumn{3}{P{20mm}}{N}&\\
    $N$ & $L_\infty$ Error & Order  \\
    \hline \hline
    %\cline{2-3}
    160&3.815$\cdot 10^{-4}$&---\\
    320&9.391$\cdot 10^{-5}$&2.022\\
    640&2.330$\cdot 10^{-5}$&2.011\\
    1280&6.116$\cdot 10^{-6}$&1.930\\
    \hline
\end{tabular}
\end{center}
\caption{Error analysis when using the local parameterization 
with a three-point stencil when solving \eqref{eq:LemniscatePDE}.}
\label{tab:Lemniscate}
\end{table}

\begin{figure}[h!]
	\centering
	\includegraphics[scale = 0.25]{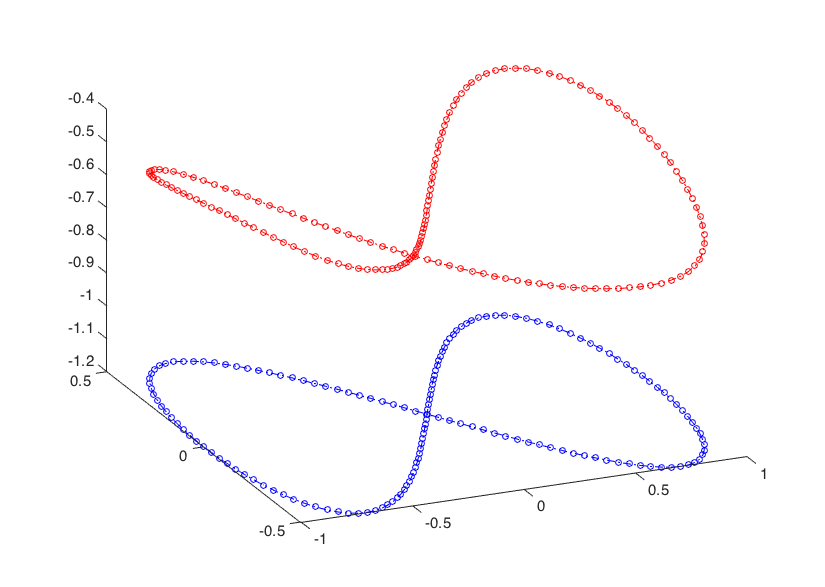}
	\includegraphics[scale = 0.25]{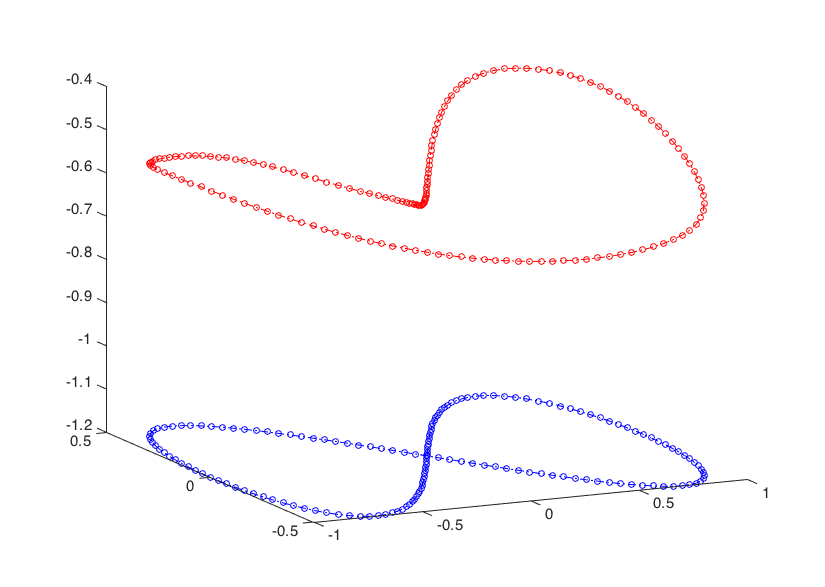} \\
	(a) \hspace{2.77in} (b)
	\caption{Solution (red) for $-\Delta u + u = x^2+xy-1$ on $x^4 - x^2 + y^2 = 0$ (blue) with $N = 160$ points using a three-point stencil, where (a) and (b) are different views of the same solution.}\label{Fig:Figure8}
\end{figure}
\end{example}

When the local degree of a local irreducible component 
is more than $1$, one can, for example,
use a truncated Puiseux series expansion
where the coefficients can be computed using numerical algebraic geometry.
Moreover, by reparameterizing (e.g., see \cite[\S~10.2.2]{SW05}),
the Puiseux series expansion is transformed into a power series
expansion and thus one can use a truncated power series expansion.  
Such a truncated expansion yields an approximation
of a local parameterization of the local irreducible component
near the singularity.  Then, one can use a discretization
of~\eqref{Eq: ODE_1d} with this approximate local parameterization 
near the singularity and use a local tangential parameterization away from the singularity.  

\begin{example}\label{ex:Cardiod}
Consider the following problem
\begin{equation}\label{Eq:Cardioid}
	-\Delta u + u = f(x,y) \qquad \hbox{~on~} (x^2+y^2)^2 + 4x(x^2+y^2) - 4y^2 = 0,
\end{equation}  
\noindent
where $f(x,y) = (3607x^2 - 224xy^2 + 7662x - 53y^2 - 973)/(196x^2 + 616x + 196y^2 + 1112)$.  The curve is called a cardioid (shown in blue
in Figs.~\ref{Fig:CardioidApprox} and~\ref{Fig:Cardioid}) which has a 
locally irreducible cusp at the origin of local degree 2.  
The choice of $f$ was selected so that \eqref{Eq:Cardioid}
has an exact solution of $u(x,y) = x + x^2$ which is
used for error analysis
provided in Table~\ref{tab:Cardioid}.
In particular, to demonstrate higher-order
methods, we used an eighth-order method with a local tangential 
approximation away from the singularity.  Near the singularity,
we approximated $x(y)$ so that $F(x(y),y) = 0$.  Since
$x(y)$ is a Puiseux series where the denominator is $3$, reparameterizing $y=s^3$ yields that $x(s)$ is power series in $s$ with the first few
terms being 
$$x(s) = s^2 - \frac{5}{12}s^4 - \frac{1}{16}s^6 - \frac{91}{5184}s^8 + \cdots.$$
To ensure more than enough accuracy, 
we used a degree 58 expansion which is pictorially shown
in Fig.~\ref{Fig:CardioidApprox}
coupled with a tenth-order discretization at the singularity.
Figure~\ref{Fig:Cardioid} shows the numerical solution of \eqref{Eq:Cardioid} computed using $N = 60$ points.
Since this computation was performed using double precision, 
the value of $N$ needs to be large enough to show convergence of the method but small enough to avoid numerical ill-conditioning.  
\vspace{-5mm}
\begin{figure}[h!]
	\centering
	\includegraphics[scale = 0.2]{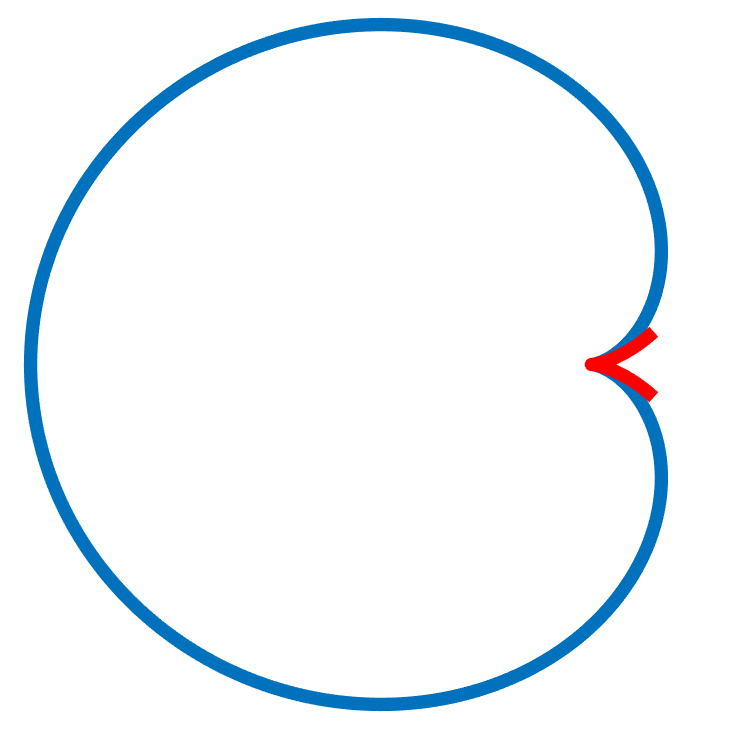}
	\caption{Cardioid (blue) along with approximation (red)
     near the singularity at the origin.}\label{Fig:CardioidApprox}
\end{figure}
\vspace{-8mm}
\begin{table}[h!]
\begin{center}
\begin{tabular}{P{20mm} | P{20mm} | P{20mm}}
    \hline \hline
    %\multicolumn{3}{P{20mm}}{N}&\\
    $N$ & $L_\infty$ Error & Order  \\
    \hline \hline
    %\cline{2-3}
    60&2.826$\cdot 10^{-5}$&---\\
    80&2.348$\cdot 10^{-6}$&8.648\\
    100&3.656$\cdot 10^{-7}$&8.334\\
    120&8.196$\cdot 10^{-8}$&8.202\\
    140&2.339$\cdot 10^{-8}$&8.134\\
    \hline
\end{tabular}
\end{center}
\caption{Error analysis for solving \eqref{Eq:Cardioid}.}
\label{tab:Cardioid}
\end{table}
%\vspace{-13mm}
\begin{figure}[h!]
	\centering
	\includegraphics[scale = 0.5]{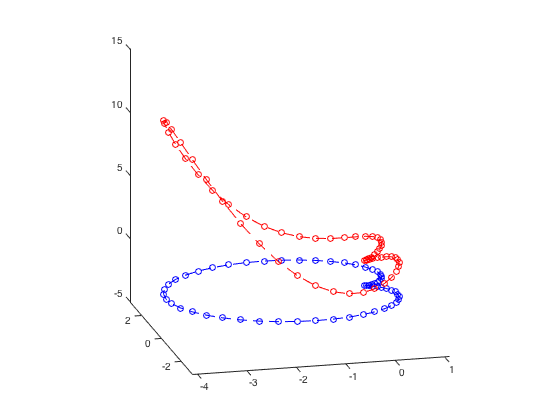}
	\caption{Solution (red) for \eqref{Eq:Cardioid} of a cardioid (blue) using $N = 60$ points.}\label{Fig:Cardioid}
\end{figure}
\end{example}
%\vspace{-5mm}

\section{Local parameterization for surfaces}\label{Sec:2d}

For smooth surfaces with a known global parameterization, 
there exists well-studied methods to solve~\eqref{Eq: ellip_pde}
as highlighted in the Introduction.
As in Section~\ref{Sec:1d} when considering curves,
we focus on the case when there is no readily available
global parameterization.

\subsection{Surface decomposition using numerical algebraic geometry}

The extension of Section~\ref{sec:DecompNAG}
to a surface is a cellular decomposition consisting
of finitely many faces $F$, which are portions of
the surface diffeomorphic to a rectangle in $\bR^2$,
along with edges $E$ and vertices $V$.  
In particular, the boundary of each face consists of finitely many 
edges, each of which has a vertex at each end.

\begin{example}\label{ex:SurfaceDecomp}
A cellular decomposition of a sphere consisting
of 2 vertices, 2 edges, and 2 faces is illustrated in Figure~\ref{fig:SphereSurfaceDecomp}. 

\begin{figure}[!ht]
    \centering
    \includegraphics[scale=0.6]{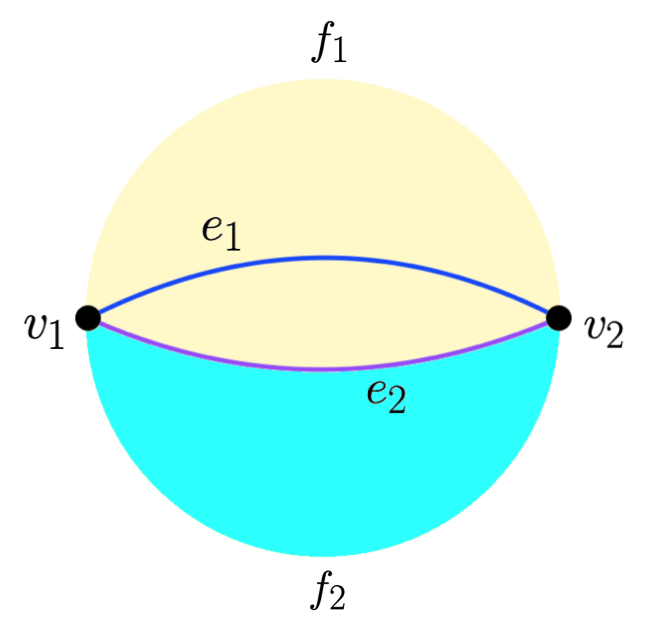}
    \caption{Surface decomposition for a sphere
    with vertices $v_i$, edges $e_j$, and faces $f_k$.}
    \label{fig:SphereSurfaceDecomp}
\end{figure}
\end{example}

A numerical cellular decomposition, as first described in~\cite{AlmostSmoothDecomp}, represents each
face by an interior point along with a homotopy
that permits the tracking along the face starting from the
interior point.  The same holds for edges as summarized in 
Section~\ref{sec:DecompNAG}.  

\subsection{Tangential parameterization at smooth points}\label{sec:LocalTangent2}

By simply adapting the approach in Section~\ref{sec:localTangent}
from a local parameterization based on the tangent line for a curve
to a local parameterization based on the tangent plane for a surface,
the following obtains an analog of Theorem~\ref{Lem: 1d_g=1}
for the surface case.

Suppose that $p$ is a smooth point on the surface $\Omega\subset\bR^n$ 
such that 
$w_1,w_2\in\bR^n$ span the tangent space with $w_i\cdot w_j = \delta_{ij}$.
Hence, the tangent space is parameterized by
$\ell(t) = p+t_1w_1+t_2w_2$.  
Let $\alpha:\Omega\rightarrow\bR^2$ where 
$\alpha(q) = ((q-p)\cdot w_1,(q-p)\cdot w_2)$.
On $\Omega$ locally nearly $p$, $\alpha$ has an inverse, say, $X(t) = X(t_1,t_2)$ where $X(0) = p$.

\begin{theorem}\label{Lem: 2d_g=1}
With the setup described above, $\frac{\partial X(0)}{\partial t_i} = w_i$.
Moreover, for corresponding metric tensor $G(t)$, 
\eqref{Eq: LB_local_coord} becomes
$$\Delta u(0) = \frac{\partial^2 u(0)}{\partial t_1^2} + \frac{\partial^2 u(0)}{\partial t_2^2}.$$
\end{theorem}
\begin{proof}
The corresponding system that $X(t)$ satisfies is
$$\left[\begin{array}{c} 
(X(t)-p)\cdot w_1 - t_1 \\
(X(t)-p)\cdot w_2 - t_2 \\
F(X(t)) 
\end{array}\right] = 0.
$$
By the implicit function theorem,
$$\left[\begin{array}{cc}
\frac{\partial X(t)}{\partial t_1} & \frac{\partial X(t)}{\partial t_2}
\end{array}\right] = - \left[\begin{array}{c}
w_1^T \\ w_2^T \\ JF(X(t))
\end{array}\right]^{-1} \left[\begin{array}{cc} -1 & 0 \\ 0 & -1 \\ 0 & 0
\end{array}\right].$$
Since $JF(X(0))w_i = 0$ and $w_i\cdot w_j = \delta_{ij}$,
one has
$\frac{\partial X(0)}{\partial t_i} = w_i$
and $w_i\cdot \frac{\partial X(t)}{\partial t_j} = \delta_{ij}$.
Hence, 
\begin{equation}\label{eq:2ndDeriv}
w_i\cdot \frac{\partial^2 X(t)}{\partial t_j \partial t_k} = 0.    
\end{equation}

By definition, the metric tensor and its inverse are
$$G(t) = \left[\begin{array}{cc}
\frac{\partial X(t)}{\partial t_1}\cdot \frac{\partial X(t)}{\partial t_1} 
& 
\frac{\partial X(t)}{\partial t_1}\cdot \frac{\partial X(t)}{\partial t_2} \\[0.1in]
\frac{\partial X(t)}{\partial t_2}\cdot \frac{\partial X(t)}{\partial t_1} 
& 
\frac{\partial X(t)}{\partial t_2}\cdot \frac{\partial X(t)}{\partial t_2}
\end{array}\right]
\hbox{~and~}
G^{-1}(t) = \frac{1}{g(t)}\left[\begin{array}{cc}
\frac{\partial X(t)}{\partial t_2}\cdot \frac{\partial X(t)}{\partial t_2} 
& 
-\frac{\partial X(t)}{\partial t_2}\cdot \frac{\partial X(t)}{\partial t_1} \\[0.1in]
-\frac{\partial X(t)}{\partial t_1}\cdot \frac{\partial X(t)}{\partial t_2} 
& 
\frac{\partial X(t)}{\partial t_1}\cdot \frac{\partial X(t)}{\partial t_1}
\end{array}\right]
$$
where $g(t) = \det G(t)$.  Hence, $G(0) = G^{-1}(0) = I_2$ and $g(0) = 1$.  
Moreover, if follows from \eqref{eq:2ndDeriv}
that $\frac{\partial g(0)}{\partial t_i} = 0$
and $\frac{\partial g^{ij}(0)}{\partial t_i} = 0$
where $g^{ij}(t)$ is the $(i,j)$-entry of $G^{-1}(t)$
so the result follows.
\end{proof}

\begin{remark}
With appropriate changes to the setup
and following a similar proof, Thm.~\ref{Lem: 2d_g=1} 
extends to smooth points on $d$-folds
in $\bR^n$.  We do not consider $d>2$ here since
it remains an open problem to compute
a numerical cell decomposition using numerical algebraic geometry
for $d>2$.
\end{remark}

\begin{example}\label{ex:EllipsoidLocal}
To illustrate, consider the ellipsoid
$x^2+10(y^2+z^2)=1$ at
$$p=\left[\begin{array}{c} 1 \\ 0 \\ 0 \end{array}\right]
\hbox{~~with~~}
w_1=\left[\begin{array}{c} 0 \\ 1 \\ 0 \end{array}\right]
\hbox{~~and~~}
w_2=\left[\begin{array}{c} 0 \\ 0 \\ 1 \end{array}\right]
\hbox{~~so that~~}
X(t_1,t_2) = \left[\begin{array}{c}
\sqrt{1-10(t_1^2+t_2^2)} \\ t_1 \\ t_2
\end{array}\right].
$$
Clearly, $X(0,0) = p$.  Since $\frac{\partial X_1(t_1,t_2)}{\partial t_i} = \frac{-10 t_i}{X_1(t_1,t_2)}$, it is clear
that $\frac{\partial X(0,0)}{\partial t_i} = w_i$.
Moreover, for $(t_1,t_2)$ near the origin, 
\eqref{Eq: LB_local_coord} becomes
$$\begin{array}{rcl}
\Delta u(t_1,t_2) &=& 
\sqrt{\frac{1-10(t_1^2+t_2^2)}{1+90(t_1^2+t_2^2)}}\cdot \left[\frac{\partial}{\partial t_1}
\left(
\sqrt{\frac{1+90(t_1^2+t_2^2)}{1-10(t_1^2+t_2^2)}}
\left(
\frac{1 - 10t_1^2 + 90t_2^2}{1 + 90(t_1^2 + t_2^2)}
\frac{\partial u}{\partial t_1}
- \frac{100t_1t_2}{1 + 90(t_1^2 + t_2^2)}
\frac{\partial u}{\partial t_2}
\right)
\right) + \right.\\
& & \hspace{1in}\left.
\frac{\partial}{\partial t_2}
\left(
\sqrt{\frac{1+90(t_1^2+t_2^2)}{1-10(t_1^2+t_2^2)}}
\left(
\frac{1 + 90t_1^2 - 10t_2^2}{1 + 90(t_1^2 + t_2^2)}
\frac{\partial u}{\partial t_2}
- \frac{100t_1t_2}{1 + 90(t_1^2 + t_2^2)}
\frac{\partial u}{\partial t_1}
\right)
\right)
\right]
\end{array}
$$
which yields $\Delta u(0,0) = \frac{\partial^2 u(0,0)}{\partial t_1^2} + \frac{\partial^2 u(0,0)}{\partial t_2^2}$
in accordance with Thm.~\ref{Lem: 2d_g=1}.
\end{example}

From an unstructured mesh of points on the surface,
one can easily construct a local discretization
of $\Delta u$ at each grid point with respect
to the local tangential parameterization
yielding a linear system to solve as in the 
curve case.

\begin{example}\label{ex:Ellipsoid}
Consider the following problem
\begin{equation}\label{Eq:ellipsoid}
	-\Delta u + u = f(x,y,z;a) \qquad \hbox{~on~} x^2+a(y^2+z^2)=1
\end{equation}  
\noindent
where $a \in \mathbb{R}_{>0}$ and $f(x,y,z;a) = \frac{1}{3(a + x^2(1-a))^2}\left[ax(2a+1) + x^3(1-a)(3a+x^2(1-a))\right]$. 
The surface is an ellipsoid (shown in Fig.~\ref{Fig:Ellipsoid})
and the choice of $f$ 
was selected so that \eqref{Eq:ellipsoid} has an exact solution
of $u(x,y,z) = x/3$ which is used for error analysis.
In particular, using a roughly uniform grid of 
size $N^2$ on the ellipsoid
with the local tangential parameterization, 
the results are summarized in Table~\ref{tab:ellipsoid}
using a nine-point stencil
for various choice of $a$. 
Figure~\ref{Fig:Ellipsoid} shows the solution of \eqref{Eq:ellipsoid} computed when $N = 40$.

\begin{table}[h!]
\begin{center}
\begin{tabular}{P{15mm} | P{20mm} | P{20mm} | P{20mm}}
    \hhline{~|===}
    %\multicolumn{3}{P{20mm}}{N}&\\
    & $N$ & $L_\infty$ Error & Order  \\
    \hline \hline
    %\cline{2-3}
    \multirow{4}{*}{a = 1}&20&3.462$\cdot 10^{-2}$&---\\
    &40&8.655$\cdot 10^{-4}$&2.000\\
    &80&2.164$\cdot 10^{-4}$&2.000\\
    &160&5.410$\cdot 10^{-5}$&2.000\\
    \hline
    \multirow{4}{*}{a = 10}&20&2.337$\cdot 10^{-2}$&---\\
    &40&5.756$\cdot 10^{-3}$&2.022\\
    &80&1.435$\cdot 10^{-3}$&2.004\\
    &160&4.056$\cdot 10^{-4}$&1.823\\
    \hline
    \multirow{4}{*}{a = 50}&20&3.219$\cdot 10^{-2}$&---\\
    &40&1.326$\cdot 10^{-2}$&1.279\\
    &80&3.606$\cdot 10^{-3}$&1.879\\
    &160&9.129$\cdot 10^{-4}$&1.982\\
\end{tabular}
\end{center}
\caption{Comparison of using
the local tangential parameterization 
on a nine-point stencil with varying values of $a$
when solving \eqref{eq:Ellipse_comparison}.}
\label{tab:ellipsoid}
\end{table}

\begin{figure}[h!]
	\centering
	\includegraphics[scale = 0.25]{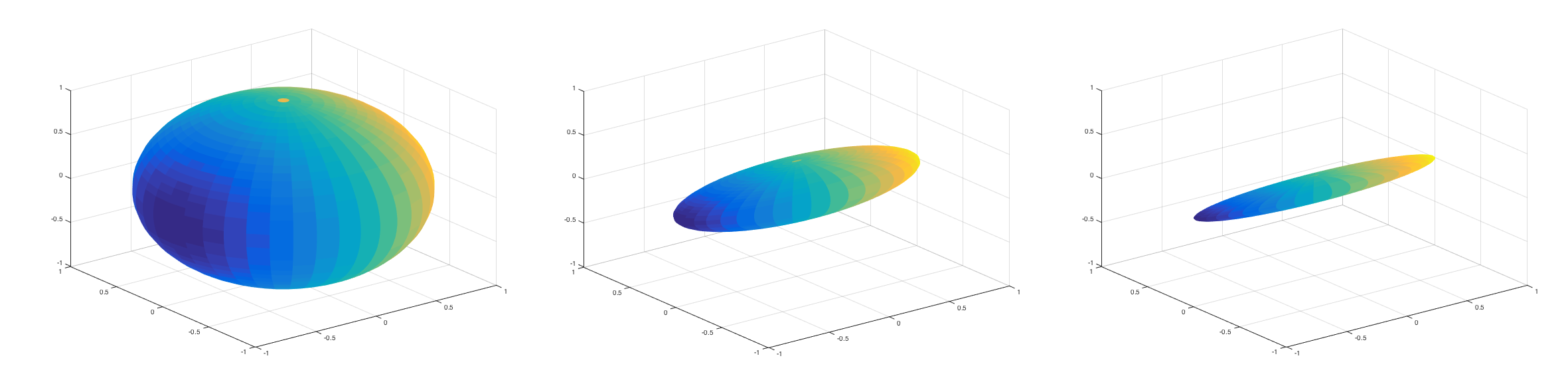} \\
	(a) \hspace{1.77in} (b)\hspace{1.77in}  (c)
	\caption{Solution of \eqref{Eq:ellipsoid} with $N = 40$ for (a) a = 1, (b) a = 10, and (c) a = 50.}\label{Fig:Ellipsoid}
\end{figure}
\end{example}

\subsection{Local parameterization near singularities}

For almost smooth surfaces, there are only
finitely many singular points and thus each are isolated.
As with the curve case, one first computes 
a local irreducible component at each singular point
since the value of $u$ at a singular point could
be different along different local irreducible components.
If a local irreducible component has local degree $1$,
it is locally diffeomorphic to a well-defined tangent 
plane for which a local tangential parameterization
from Section~\ref{sec:LocalTangent2} can be used.
For local irreducible components of higher local
degree, one can use a local parameterization 
(or an approximation of one)
to discretize near the singularity for each
each local irreducible component.

\begin{example}
Consider the following problem
\begin{equation}\label{Eq:HornTorus}
	-\Delta u + u = x \qquad \hbox{~on~} (x^2 + y^2 + z^2)^2 - 4(x^2+y^2) = 0
\end{equation}   
where the surface is a called a horn torus (shown in Fig.~\ref{Fig:HornTorus}).
The horn torus is almost smooth with a singularity
at the origin.
Using an approximately uniform grid of $N^2$ points,
the local tangential parameterization was used
away from the origin.
The surface is locally irreducible at the origin
and the following local parameterization was utilized: 
 $$x(t_1,t_2) = t_1^2\cos(t_2),~~~~y(t_1,t_2) = t_1^2\sin(t_2),
 ~~~~z(t_1,t_2) = t_1\sqrt{2-t_1^2}.
 $$
A nine-point stencil was used at all points
with the results summarized in Table~\ref{tab:HornTorus}
where the error is computed by comparing with the 
solution computed when $N=160$.
Figure~\ref{Fig:HornTorus} shows the solution of \eqref{Eq:HornTorus} 
when~$N = 40$.

%\begin{table}[h!]
%\begin{center}
%\begin{tabular}{P{20mm} | P{40mm}}
%    \hline \hline
    %\multicolumn{3}{P{20mm}}{N}&\\
%    N & $L_\infty$ Error  \\
%    \hline \hline
    %\cline{2-3}
%    20&0.5612\\
%    40&0.4512\\
%    80&0.0960\\
%    160&0.0537\\
%    \hline
%\end{tabular}
%\end{center}
%\caption{Error analysis for the problem defined by Eq.~\ref{Eq:HornTorus} solved using the local tangential parameterization away from the singularity and a local reparameterization near the singularity.}
%\label{tab:HornTorus}
%\end{table}

\begin{table}[h!]
\begin{center}
\begin{tabular}{P{20mm} | P{40mm}}
    \hline \hline
    %\multicolumn{3}{P{20mm}}{N}&\\
    $N$ & $L_\infty$ Error  \\
    \hline \hline
    %\cline{2-3}
    20&2.235$\cdot 10^{-2}$\\
    40&1.464$\cdot 10^{-3}$\\
    80&3.075$\cdot 10^{-4}$\\
    160& -- \\
    \hline
\end{tabular}
\end{center}
\caption{Error analysis for solving \eqref{Eq:HornTorus}.}
\label{tab:HornTorus}
\end{table}

\begin{figure}[h!]
	\centering
	\includegraphics[scale = 0.3]{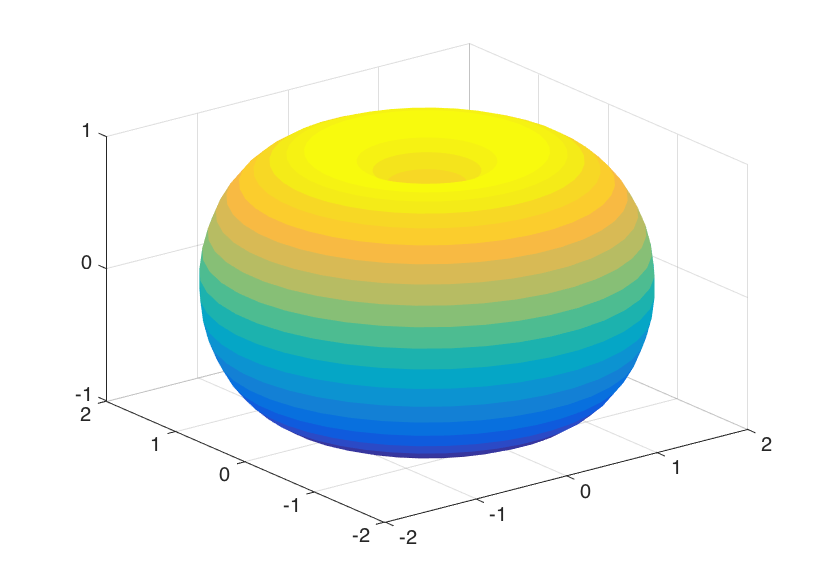}
	\caption{Solution for $-\Delta u + u = x$ on $(x^2 + y^2 + z^2)^2 - 4(x^2+y^2) = 0$ when $N = 40$.}\label{Fig:HornTorus}
\end{figure}

\end{example}

%\section*{Data availability}
%Enquiries about data availability should be directed to the authors.

%\section*{Conflict of interest}
%The authors declare that they have no conflict of interest.

\newcommand{\noopsort}[1]{}\def\cprime{$'$}

\end{document}